\documentclass[a4paper,11pt]{scrartcl}

\usepackage{graphicx}
\usepackage[utf8]{inputenc}
\usepackage[english]{babel}
\usepackage{amsmath,amssymb,enumerate,url,tabularx,stmaryrd,mathrsfs}

\usepackage{todonotes,comment,afterpage}
\usepackage{subcaption}
\usepackage{framed}
\usepackage{mathtools}
\usepackage{fullpage}
\usepackage[title]{appendix}

\usepackage{amsthm}
\usepackage{graphicx}
\usepackage{hyperref}

\theoremstyle{plain}
\newtheorem{theorem}{Theorem}[section]
\newtheorem{lemma}[theorem]{Lemma}
\newtheorem{corollary}[theorem]{Corollary}
\newtheorem{proposition}[theorem]{Proposition}

\newtheorem{definition}[theorem]{Definition}

\newtheorem{remark}[theorem]{Remark}

\numberwithin{equation}{section}
\numberwithin{figure}{section}
\numberwithin{table}{section}

\newcommand{\R}{\mathbb{R}}
\newcommand{\C}{\mathbb{C}}
\newcommand{\N}{\mathbb{N}}
\newcommand{\Z}{\mathbb{Z}}
\newcommand{\bigO}{\mathcal{O}}

\newcommand{\abs}[1]{\left|#1\right|}
\newcommand{\norm}[1]{\left\|#1\right\|}

\newcommand{\sinc}{\operatorname{sinc}}

\newcommand{\dom}{\operatorname{dom}}

\newcommand{\atan}{\tan^{-1}}
\newcommand{\asin}{\sin^{-1}}

\newcommand{\asinh}{\sinh^{-1}}

\def\ii{i}

\def\id{\mathrm{I}}

\def\LL{\mathcal{L}}
\def\CC{\mathcal{C}}

\newcommand{\HHx}{\mathbb{V}_{h}}

\renewcommand{\Re}{\operatorname{Re}}
\renewcommand{\Im}{\operatorname{Im}}

\def\Nquad{\mathcal{N}_{\textrm{q}}}
\def\kquad{k}

\newcommand{\Dexp}{D^{\mathrm{exp}}_\delta}
\newcommand{\Dt}{D_{d(\theta)}}
\newcommand{\Ht}{H_\theta}

\usepackage{tikz}
\usetikzlibrary{shapes}
\usepackage{pgfplots}
\usetikzlibrary{calc}

\usetikzlibrary{positioning}
\usetikzlibrary{arrows}
\usetikzlibrary{calc}
\usetikzlibrary{intersections, pgfplots.fillbetween,patterns}
\usepgfplotslibrary{colorbrewer}

\pgfplotsset{compat=1.15}
\iffalse 
\newcommand{\includeTikzOrEps}[1]{\tikzexternalenable \tikzsetnextfilename{#1}  {\include{figures/#1}} \tikzexternaldisable}

\usetikzlibrary{external}
\tikzexternaldisable

\else
\newcommand{\includeTikzOrEps}[1]{\includegraphics{figures_pdf/#1}}
\fi

\newcommand{\psits}{\psi_{\sigma,\theta}}
\newcommand{\varphits}{\varphi_{\sigma,\theta}}
\newcommand{\varphith}{\varphi_{\frac{1}{2},\theta}}

\newcommand{\psito}{\psi_{1,\theta}}
\newcommand{\psith}{\psi_{\frac{1}{2},\theta}}

\title{Double exponential quadrature for fractional diffusion }

\author{ Alexander Rieder\thanks{ Fakult\"at f\"ur Mathematik, University of Vienna, 1090 Vienna, Austria,
    \href{mailto:alexander.rieder@univie.ac.at}{alexander.rieder@univie.ac.at}}}
\date{\today}

\begin{document}
\maketitle
\begin{abstract}
  We introduce a novel discretization technique for both elliptic and parabolic fractional
  diffusion problems based on double exponential quadrature formulas and the Riesz-Dunford functional calculus.
  Compared to related schemes, the new method provides faster convergence with fewer parameters that need
  to be adjusted to the problem. The scheme takes advantage of any additional smoothness
  in the problem without requiring a-priori knowledge to tune parameters appropriately.
  We prove rigorous convergence results for both, the case of finite regularity data as well as
  for data in certain Gevrey-type classes. We confirm our findings with numerical tests.
\end{abstract}

\section{Introduction}
The study of processes governed by fractional linear operators has gathered
significant interest over the last few  years~\cite{bv16,v17,whatis} with
applications ranging from physics~\cite{ab17b} to image processing~\cite{gh15,go08,ab17b},
inverse problems~\cite{kr19} and more. See~\cite{collection_of_applications}
for an overview of applications in different fields.

There are multiple (non-equivalent) ways of defining fractional powers of operators. We mention
the integral fractional Laplacian and the spectral definition~\cite{whatis}.
In this paper, we focus the spectral definition which is equivalent to the functional calculus definition.

For discretization of such problems, both  stationary and time dependent, multiple approaches have
been presented. A summary of the most common can be found in~\cite{bbnos18,whatis}. They can be broadly
distinguished into three categories. First, there are approaches
which try to apply the methods of finite- and boundary element methods to  integral
formulation of the fractional Laplacian~\cite{ab17,abh19}. The second class of methods
uses the Caffarelli-Silvestre extension to reformulate the problem as a PDE posed in one additional
spatial dimension. This problem is then treated by standard finite element techniques
\cite{nos15,pde_approach_apriori,pde_frac_parabolic,tensor_fem,mpsv17,hp_for_heat}.
The third big class of discretization schemes, and the one our new scheme is part of,
was first introduced in~\cite{bp15} and later extended to more general operators~\cite{blp_accretive}
and time dependent problems~\cite{bonito_pasciak_parabolic,blp17,MR20_hp_sinc}. They are based on
the Riesz-Dunford calculus (sometimes also referred to as Dunford-Taylor or Riesz-Taylor) and employ
a $\sinc$ quadrature scheme to discretize the appearing contour integral.
$\sinc$ quadrature, and overall $\sinc$-based numerical methods are less well known than
their polynomial based counterparts, but provide rapidly converging schemes~\cite{stenger03,sinc_book}
with very easy implementation. The quadrature relies on appropriate coordinate transforms in order to
yield analytic, rapidly decaying integrands  over the real line and then discretization using
the trapezoidal quadrature rule.
In~\cite{tm73} it was realized that by adding an additional $\sinh$-transformation, it is possible to get
an even faster convergence for certain integrals, namely instead of convergence of the form
$e^{-\sqrt{\Nquad}}$, it is possible to get rid of the square root and obtain rates of the form $e^{-\frac{\Nquad}{\ln{\Nquad}}}$.
Further developments in this direction are summarized in~\cite{mori91}. Such schemes are commonly referred to
as double exponential quadrature or $\sinh$-$\operatorname{tanh}$ quadrature.

In this paper we investigate whether the discretization of the Riesz-Dunford integral can benefit from using
a double exponential quadrature scheme instead of the more established $\sinc$-quadrature.
We present a scheme that retains all the advantages of \cite{blp_accretive,bonito_pasciak_parabolic,blp17}
while delivering improved convergence rates.
Namely, the scheme is very easy to implement if a solver for elliptic finite element problems is available.
It is almost trivially parallelizable, as the main cost consists of solving a series of independent
elliptic problems. In addition, it provides (compared to $\sinc$-methods)
superior accuracy over a wide range of applications and does not require subtle tweaking of parameters
in order to get good performance. Instead it will automatically pick up any additional smoothness
of the underlying problem to give improved convergence. Since for each quadrature
point an elliptic FEM problem needs to be solved, reducing the number of quadrature points
greatly increases performance of the overall method.

The paper is structured as follows.
After fixing the model problem and notation in Section~\ref{sect:model_problem},
Section~\ref{sec:de_formulas} introduces the
double exponential formulas in an  abstract way and we collect some known properties. In addition,
we provide
one small convergence result which, to our knowledge, has not yet appeared in the literature;
we show that the double exponential formulas at least provide comparable convergence of order $e^{-\sqrt{\Nquad}}$
even without requiring additional analyticity compared to standard $\sinc$ methods.
In Section~\ref{sect:elliptic}, we look at the case of a purely elliptic problem without time dependence.
It will showcase the techniques used and provide the building block for the  more involved problems later on.
In Section~\ref{sect:parabolic}, we then consider what happens if we move into the time-dependent regime.
Section~\ref{sect:numerics} provides extensive numerical evidence supporting the theory. We also compare
our new method to the standard $\sinc$-based methods.
Finally, Appendix~\ref{sect:properties_fo_psi} collects some properties of the coordinate transform
involved. The proofs and calculations are elementary but somewhat lengthy and thus have been relegated
to the appendix in order to not impact readability of the article.

Throughout this work we will encounter two types of error terms. For those of the
form $e^{-\frac{\gamma}{k}}$ we will be content with not working out the constants $\gamma$ explicitly. For the
more important  terms of the form $e^{-\frac{\gamma'}{\sqrt{k}}}$ we will
derive explicit constants $\gamma'$ which prove sharp in several examples of Section~\ref{sect:numerics}.

We close with a remark on notation. Throughout this text, we write $A \lesssim B$ to mean
that there exists a constant $C>0$, which is independent of the main quantities of interest
like number of quadrature points $\Nquad$ or step size $k$ such that
$A \leq C B$. The detailed dependencies of $C$ are specified in the context. We write
$A \sim B$ to mean $A \lesssim B$ and $B \lesssim A$.

\paragraph{Acknowledgments:} The author would like to thank J.~M.~Melenk for the many fruitful discussions on the
topic.
Financial support was provided by the Austrian Science Fund (FWF) through the 
special research program ``Taming complexity in partial differential systems'' (grant SFB F65)
and the project P29197-N32.

\subsection{Model problem and notation}
\label{sect:model_problem}

In this paper, we consider problems of applying holomorphic functions $f$ to self-adjoint operators, for example the Laplacian.
The two large classes of problems treated in this paper stem from the study of fractional diffusion problems, both in the stationary
as well as in the transient version.
Given a bounded Lipschitz domain $\Omega$, we consider the self adjoint operator
$$
\LL u:=-\operatorname{div}(\mathfrak{A} \nabla u) + \mathfrak{c} u,
$$
where $\mathfrak{A} \in L^{\infty}(\Omega;\R^{d\times d})$ is uniformly SPD and $\mathfrak{c} \in L^{\infty}(\Omega)$ satisfies $\mathfrak{c}\geq 0$ almost everywhere.
Unless otherwise specified, the domain $\operatorname{dom}(\LL)$ is always taken to include homogeneous
Dirichlet boundary conditions.

Given the eigenvalues and eigenfunctions of $\LL$ equipped with homogeneous Dirichlet boundary conditions denoted by $(\lambda_j,v_j)_{j=0}^{\infty}$,
we define the spaces
\begin{align}
  \label{eq:def_HH}
  \mathbb{H}^{\beta}(\Omega):=\Big\{ u \in L^2(\Omega): \norm{u}_{\mathbb{H}^\beta(\Omega} <\infty \Big\}
  \qquad\text{with} \qquad
  \norm{u}_{\mathbb{H}^{\beta}(\Omega)}:=\sum_{j=0}^{\infty}{\lambda_j^{\beta} \abs{(u,v_j)_{L^2(\Omega)}}^2}.  
\end{align}
One natural way of defining a functional calculus for the operator $\LL$ is based on the spectral decomposition. 
\begin{definition}[Spectral calculus]
  \label{def:spectral_calculus}
  Let $\mathcal{O} \subseteq \R_+$ such that $\sigma(\LL) \subseteq \mathcal{O}$.
  Let $g: \mathcal{O} \to \C$ be continuous with $\abs{g(z)}\lesssim (1+\abs{z})^{\mu}$ for $\mu \in \R$. We define
  for $u \in \mathbb{H}^{2\mu}(\Omega)$:
  \begin{align*}
    g(A)u:=\sum_{j=0}^{\infty}{g(\lambda_j) \big(u,v_j\big)_{L^2(\Omega)} \, v_j }.
  \end{align*}
\end{definition}

An alternative definition for holomorphic functions, which will prove more useful for approximation is 
given in the following Definition. It can be shown, see also~\cite[Section 2]{blp17},
that the operators resulting
from Definition~\ref{def:spectral_calculus} and Definition~\ref{def:riesz-dunford} coincide.
\begin{definition}[Riesz-Dunford calculus]
  \label{def:riesz-dunford}
  Fix parameters $\sigma\in \{1/2,1\}$, $\theta \geq 1$ and $\kappa>0$.
  Let $\mathcal{O} \subseteq \C$ such that $\C_{+}:=\{ z \in \C: \Re(z)>0\} \subseteq \mathcal{O}$.
  Let $g: \mathcal{O} \to \C$ be holomorphic with $\abs{g(z)}\lesssim (1+\abs{z})^{-\mu}$
  for $\mu \geq 0$. We define
  \begin{align}
    \label{eq:riesz_dunford}
    g(A):=\frac{1}{2\pi\ii}\int_{\CC}{g(z)\big(\LL-z\big)^{-1} \,dz},
  \end{align}
  where the integral is taken in the sense of Riemann, and $\CC$ is the smooth path
  $$
  \CC:=\Big\{\kappa\big(\cosh(\sigma w) + \theta \sinh(w)\big) \quad \text{ for } w \in\R \Big\}.
  $$
  The parameter $\kappa>0$ is taken sufficiently small such that the spectrum of $\LL$ satisfies
  $\kappa<\sigma(\LL)$.
\end{definition}
\begin{remark}
  The choice of path in Definition~\ref{def:riesz-dunford} is somewhat arbitrary. It is only
  required to encircle the spectrum of $\LL$ with winding number $1$.
  Throughout this paper, we will only ever use
  the same path and thus make it part of our definition.
\end{remark}
\begin{remark}
  All our results easily generalize to  more general positive definite and self-adjoint
  operators $\mathcal{L}$ over a general Hilbert space $\mathcal{X}$, as long as an eigen-decomposition
  is available. The $\mathbb{H}^\beta$-norms in \eqref{eq:def_HH} are
  then defined using the $\mathcal{X}$ inner product.
\end{remark}

\section{Double exponential formulas}
\label{sec:de_formulas}
In this section, we analyze the quadrature error when applying a double exponential formula for discretizing certain integrals.
The main role will be played by the following coordinate transform:
\begin{align}
  \label{eq:def_psi}
  \psi_{\theta,\sigma}(y):=\kappa\Big[\cosh\Big(\frac{\sigma \pi}{2} \sinh(y)\Big) + \ii \theta \sinh\Big(\frac{\pi}{2} \sinh(y)\Big)\Big].
\end{align}
We will focus on the cases $\sigma \in \big\{\frac{1}{2},1 \big\}$ and $\theta \geq 1$. $\kappa$ is again
taken sufficiently small as in Definition~\ref{def:riesz-dunford}.

For $\theta \geq 1$, $\delta>0$ we define the sets
\begin{align}
  \label{eq:def_dt_dexp}
  \Dt:=\Big\{z \in \C: \abs{\Im(z)} < d(\theta) \Big\}, \qquad \text{and} \qquad 
  \Dexp:=\Big\{z \in \C: \abs{\Im(z)} < \delta e^{-\abs{\Re(z)}}  \Big\},
\end{align}
where for each $\theta$,  $d(\theta)$ is a constant which is assumed sufficiently small.

Since all the proofs analyzing the properties of $\psits$ are elementary but somewhat lengthy and cumbersome,
they have been relegated to Appendix~\ref{sect:properties_fo_psi}. The most important
properties are, that $y\mapsto \psits(y)$ for $y\in \R$ traces the contour in the definition of the Riesz Dunford calculus
(see~Definition~\ref{def:riesz-dunford}), and that it is analytic in $\Dt$.
The other important results concern the points where $\psits$ crosses the real axis, as these points correspond to
(possible) poles in the integrand of Definition~\ref{def:riesz-dunford}. The location of these points, as well as
other important estimates are collected in Lemma~\ref{lemma:psi_hits_lambda}.
Roughly summarizing, the finitely many points satisfying $\psits(y)=\lambda$ have distance $1/\ln(\lambda)$ from the real axis. Away
from such points $\abs{\psits(y)-\lambda}\gtrsim\lambda$ holds and for $y \to \pm \infty$ the function $\psits$ behaves doubly-exponential (Lemma~\ref{lemma:psi_growth}).

Because the quadrature operator will be the main workhorse, we introduce the following notation:
\begin{definition}
  \label{def:quadrature_operators}
  Let $\LL$ be a linear (possibly unbounded) operator
  on a Banach space $\mathcal{X}$ and
  $\mathcal{O}\subseteq \C$ such that $\C_+ \subseteq \mathcal{O}$.
  For $g: \mathcal{O} \to \C$ holomorphic as in Definition~\ref{def:riesz-dunford},
  $k \geq 0$ and $\Nquad \in \N \cup \{\infty\}$ we write
  \begin{align}
    \label{eq:def_quadrature_operator}
    Q^{\LL}(g,\Nquad)u&:=
                        \frac{1}{2\pi \ii} \sum_{j=-\Nquad}^{\Nquad}{g(\psits(j k))  \psits'( j k)(\LL - \psits(j k))^{-1} u}
                        \qquad \forall u \in \mathcal{X}
  \end{align}
  and $Q^{\LL}(g):=Q^{\LL}(g,\infty)$ for the case where no cutoff is performed.  
  The quadrature error will be denoted by
  \begin{align}
    \label{eq:def_quadrature_error}
    E^{\LL}(g,\Nquad)
    &:=g(\LL)u - Q^{\LL}(g,\Nquad)u,  \qquad \forall u \in \dom(g(\LL))
  \end{align}
  where $g(\LL)$ is given via the Riesz-Dunford integral~\ref{def:riesz-dunford}.
  Again, we write $E^{\LL}(g):=E^{\LL}(g,\infty)$.
\end{definition}
\begin{remark}
  In Definition~\ref{def:quadrature_operators}, we will often work
  with the special case $\LL=\lambda$ or $\LL=\lambda \id$ for $\lambda \in \R$.
  For simplicity, in both cases we write
  $Q^{\lambda}(g,\Nquad)$, $E^{\lambda}(g,\Nquad)$ etc. It will be clear from context whether
  the scalar or operator valued operation is needed.
\end{remark}

\subsection{Abstract analysis of $\operatorname{sinc}$-quadrature }
In this section, we collect some results
on $\sinc$-quadrature formulas.
\begin{remark}
  As is common in the literature, we define the $\sinc$ function as
  $$
  \sinc(\zeta):=\begin{cases}
    \frac{\sin(\pi \zeta)}{\pi \zeta} & \zeta \neq 0 \\
    1 & \zeta=0.
  \end{cases}
  $$
\end{remark}

The following result is the main work-horse when analyzing $\operatorname{sinc}$-quadrature
schemes. In order to reduce the required notation, we use a simplified version of
\cite[Problem 3.2.6]{stenger03}.
\begin{proposition}[{Bialecki, see \cite[Problem 3.2.6 and Theorem 3.1.9]{stenger03}}]
  \label{prop:sinc_with_poles}
  We make the following assumptions on $g$:
  \begin{enumerate}[(i)]
  \item
    \label{def:sinc_functions:i}
    $g$ is a meromorphic function on the infinite strip $\Dt$.
    It is also continuous on $\partial \Dt$.
    The poles $\big(p_{\ell}\big)_{\ell=1}^{N_p}$  are all simple and located in $\Dt \setminus \R$.
  \item
    \label{def:sinc_functions:ii}
    There exists a constant $C>0$ independent of $y \in \R$ such that
    for sufficiently large $y > 0$,
    \begin{align}
      \int_{-d(\theta)}^{d(\theta)}{\abs{g(y+\ii w)} \,dw}&\leq C.
    \end{align}
  \item
    \label{def:sinc_functions:iii}
    We have
    \begin{align}
      N(g,\Dt):=\int_{-\infty}^{\infty}{\abs{g(y+\ii d(\theta)} + \abs{g(y-\ii d(\theta))}\,dy} < \infty.
    \end{align}   
  \end{enumerate}
  Denote by $\operatorname{res}(g;p_\ell)$ the residue of $g$ at $p_\ell$,
  and define
  $
  \gamma(k;p_\ell):=\frac{1}{\sin(\pi p_\ell/k)}.
  $
  
  Then there exists a generic constant, such that
  for all $k > 0$, using $s_\ell:=\operatorname{sign}(p_\ell)$:
  \begin{align}
    \label{eq:sinc_with_poles}
    \abs{
    \int_\R{g(t) \, dt} - k \sum_{n=-\infty}^{\infty}{g(k \, n)}
    - \pi \sum_{\ell=1}^{N_p} {e^{\ii  \frac{s_\ell \pi p_\ell}{k}}  \operatorname{res}(g;p_\ell)  \gamma(k;p_\ell)
      }
    }
    &\lesssim \frac{e^{-2\pi d(\theta)/k}}{1-e^{-2\pi d(\theta)/k}} N(g,\Dt). 
  \end{align}
\end{proposition}


  

Proposition~\ref{prop:sinc_with_poles} requires certain decay properties for the integrand in
a complex strip,
and thus is not always applicable.
As is shown in Appendix~\ref{sect:properties_fo_psi}, the transformation $\psits$ maps 
partly into the left-half plane. One can even show that
the real part changes sign infinitely many times when evaluating along a line
of fixed imaginary part. If we therefore consider the case when $f(z):=e^{-z}$ is the exponential function,
this means that $f \circ \psi$ is exponentially increasing in such regions. This
puts showing estimates of the form required in
Proposition~\ref{prop:sinc_with_poles}~(\ref{def:sinc_functions:iii}) out of reach.

On the other hand, Lemma~\ref{lemma:psi:properties:right_halfplane}
shows that for $\sigma=1$, restricted to the domain $\Dexp$,
the map $\psits$ stays in the right half-plane.
Here the exponential function is decreasing.
Similarly, the Mittag-Leffler function $e_{\alpha,\mu}$ is decreasing on slightly larger sectors,
allowing for the choice of $\sigma=1/2$ if $\alpha < 1$.
This motivates
the following modification of Proposition~\ref{prop:sinc_with_poles}.

\begin{lemma}
  \label{lemma:sinc_in_dexp}
  Assume that $g: \Dexp \to \C$ is holomorphic
  and is doubly-exponentially
  decreasing, i.e., there exist constants $C_g > 0$, $\mu >0,$ such that
  $g$ satisfies 
  \begin{align}
    \label{eq:sinc_in_dexp:de_requirement}
    \abs{g(y)}\leq C_g \exp\big(-\mu e^{\Re(y)}\big) \qquad \forall y \in \Dexp.
  \end{align}

  Then, for all $\varepsilon > 0$,
  there exists a constant $C>0$ which is independent of $k$, $\mu$ and $g$
  such that the following error estimate holds:
  \begin{align}
    \label{eq:sinc_in_dexp:statement}
    \abs{\int_{\R}{g(t) dt} -
    k\sum_{n=-\infty}^{\infty}{g(k\,n)}}
    &\leq C C_g   k\,\exp\Big({- \sqrt{8\pi \delta} \,\frac{\sqrt{\mu-2\varepsilon}}{\sqrt{k}}}\,\Big).
  \end{align}
\end{lemma}
\begin{proof}
  We closely follow the proof of~\cite[Theorem 2.13]{sinc_book}, but
  picking a different contour
  and later exploiting the strong decay properties of $f$.
  
  For $N \in \N$,
  set $R_N:=\Big\{ y \in \C: \abs{\Re(y)} \leq (N+\frac{1}{2}),
  \abs{\Im(y)} \lesssim \delta\, e^{-\abs{\Re(y)}}\Big\}$.
  For fixed $t \in \R$, we fix  $N$ large enough such that $t \in R_N$.
  By applying the residue theorem to the function
  $$
  h(y):=\frac{\sin(\pi \, t / k) g(y)}{(t-y) \sin(\pi y/k)},
  $$
  one can show the equality
  \begin{align*}
  g(t) -  k\sum_{n=-N}^{N}{g(n\,k) \sinc\Big(\frac{t-nk}{k}\Big)}
    &= \int_{\partial R_N}{
      \frac{\sin(\pi \, t / k)g(y)}{(t-y) \sin(\pi y/k)} \,dy
      }.
  \end{align*}
  Since asymptotically $f(t)$ decreases doubly exponentially,
  while $1/\sin(\pi y/k)$ only grows exponentially along the
  path $\{(\xi,\delta \, e^{-\xi}), \, \xi \in \R\}$, we can pass to the limit $N \to \infty$ to get the representation
  \begin{align}
    \label{eq:sinc_error_repr_pointwise}
    g(t) -  k\sum_{n=-\infty}^{\infty}{g(k\,n) \sinc\Big(\frac{t-nk}{k}\Big)}    
    &= \int_{\partial \Dexp}{
      \frac{\sin(\pi \, t / k)g(y)}{(t-y) \sin(\pi y/k)} \,dy
      }.
  \end{align}
  Integrating~(\ref{eq:sinc_error_repr_pointwise}) 
  over $\R$ and exchanging the order of integration
  gives:
  \begin{align}
    \int_{\R}{g(t) \, d\tau} -  k\sum_{n=-\infty}^{\infty}{g(k\,n)}    
    &= \int_{\partial \Dexp}{
      \frac{g(y)}{ \sin(\pi y/k)} \int_{\R}{ \frac{\sin(\pi \, t / k)}{t-y} \, dt} \, dy.
      } \nonumber \\ 
      &= -\pi \int_{\partial \Dexp}{
      \frac{g(y)}{ \sin(\pi y/k)}
      e^{\frac{\ii \operatorname{sign}(\Im(y)) \pi y}{k}} \, dy} ,
        \label{eq:sinc_error_repr_integrated}        
  \end{align}
  where in the last step we invoked \cite[Lemma 2.19]{sinc_book} to explicitly evaluate the
  integral.
  What remains to do is bound the integral on the right-hand side.
  For simplicity, we focus on the upper-right half-plane. The other cases follow
  analogously. There, we can parameterize $\partial \Dexp$ as
  $y=\xi + \ii  \delta\, e^{-\xi}$. We estimate
  \begin{align}
    \abs{\frac{g(y)}{ \sin(\pi y/k)}
    e^{\frac{\ii \operatorname{sign}{\Im(y)} \pi y}{k}}}
    &\lesssim \;
      \abs{g(y)} \frac{\exp\Big({-2\pi \delta\,\frac{ e^{-\xi}}{k}}\Big)}
      {1- \exp\Big({-2\pi \delta\,\frac{ e^{-\xi}}{k}}\Big)}\nonumber\\
    &\lesssim \;
       \abs{g(y)} k\,e^{\xi} \exp\Big({-2\pi \delta\,\frac{ e^{-\xi}}{k}}\Big) \nonumber \\
    &\stackrel{\mathclap{(\ref{eq:sinc_in_dexp:de_requirement})}}{\lesssim}
      \; C_g k\exp\Big(-\mu e^\xi + \xi - \frac{2\pi \, \delta e^{-\xi}}{k} \Big)
      \label{eq:sinc_in_dexp_proof1}
  \end{align}
  For $\varepsilon > 0$, we can absorb the linear $\xi$-term into the first
  exponential, and estimate:
  \begin{align*}
    \eqref{eq:sinc_in_dexp_proof1}
    &\lesssim \varepsilon^{-1}
      C_g k\exp\Big(-(\mu-2\varepsilon) e^\xi - \frac{2\pi \, \delta e^{-\xi}}{k} \Big)
      \exp\Big( -\varepsilon e^\xi \Big)
  \end{align*}
  where the second term will be used to regain integrability, whereas the
  first one will give us approximation quality.
  For $\xi = 0$ and $\xi \to \infty$,
  we get sufficient bounds to prove~(\ref{eq:sinc_in_dexp:statement}).
  We thus have to look for maxima of the function with respect to $\xi$ in between $(0,\infty)$.
  Due to monotonicity of the exponential, we focus on the argument and set
  $
  \tau:=e^{\xi}.
  $
  By setting its derivative to zero we get that the map
  \begin{align*}
    \tau \mapsto - (\mu - 2\varepsilon) \tau - \frac{2\pi \, \delta}{\tau \, k}
    \quad
    \text{ is maximized for }
    \quad
    \tau_{\max} = \sqrt{\frac{2\delta  \pi}{ k (\mu - 2\varepsilon)}}.
  \end{align*}
  Inserting  all this into~(\ref{eq:sinc_error_repr_integrated}),
  we get
  \begin{align*}
    \abs{\int_{\R}{g(t) \, d\tau} -  k\sum_{n=-\infty}^{\infty}{g(k\,n)}    }
    &\lesssim
      C_g k \exp\Big(- 2 \sqrt{2\pi \delta}\sqrt{\frac{\mu - 2 \varepsilon}{k}}\Big)
      \int_{0}^\infty{      
      \exp\Big(-\varepsilon e^{\abs{\Re(y)}} \Big) \, d\xi
      }\\
    &\lesssim C_g k
      \exp\Big(- 2 \sqrt{2\pi \delta}\sqrt{\frac{\mu - 2 \varepsilon}{k}}\Big).
      \qedhere
  \end{align*}
\end{proof}
\begin{remark}
  It is also possible to admit meromorphic functions with finitely many poles into
  Lemma~\ref{lemma:sinc_in_dexp},
  as long as additional error terms analogous to \eqref{eq:sinc_with_poles}
  are introduced. Since we will not need this generalization we stay in
  the analytic setting.
\end{remark}

While Lemma~\ref{lemma:sinc_in_dexp} provides a reduced rate of convergence compared
to the more-standard $\sinc$-quadrature of Proposition~\ref{prop:sinc_with_poles}
($k^{-1/2}$ vs $k^{-1} \abs{\ln(k)}$), thus removing the advantage we want to achieve by using
the double exponential transformation,
we will later consider a class of functions which
decay fast enough to allow us to tune the parameter $\mu \sim k^{-1}$ to regain almost
full speed of convergence.

Finally, we show how the transformation $\psits$ and the operator $\LL$ enter the estimates. The
next corollary also showcases how the cutoff error is controlled.
\begin{corollary}
  \label{cor:sinc_in_dexp_for_quadrature_operators}
  Let $\mathcal{O} \subseteq \C$ contain the right half-plane,
  and if $\sigma=1/2$ also a sector
  $$S_{\omega}:=\big\{z \in \C: \abs{\operatorname{Arg}(z)}\leq\omega  \big\}
  \qquad \text{for some}  \quad \omega>\frac{\pi}{2}.
  $$
  Assume that $g: \mathcal{O} \to \C$ is analytic and satisfies the polynomial bound
  $$
  \abs{g(z)} \leq C_{g} (1+\abs{z})^{-\mu} \qquad \text{for $\mu  > 0$}.
  $$
  Then, for all $0<\varepsilon<\mu/2$, $0\leq s\leq r$ and $\lambda>\lambda_0>\kappa$,
  the quadrature errors can be bounded by:  
    \begin{align*}
      \norm{E^{\LL}(g,\Nquad)}_{\mathbb{H}^{2s}(\Omega) \to \mathbb{H}^{2r}(\Omega)}
      &\leq C\, C_g  
        \Bigg[\exp\Big(- \gamma_1 \frac{ \sqrt{\mu - r +s  -2\varepsilon}}{\sqrt{k}}\Big)
        +\exp\Big(- (\mu-r+s) \gamma e^{k \Nquad}\Big)\Bigg].
    \end{align*}
    The constant $C$ is independent of $g$, $k$,$r$,$s$ and $\beta$,
    but may depend on $\varepsilon$, $\sigma$, $\theta$.
    The rate $\gamma_1$ depends on $\theta$ and $\omega$.
    $\gamma$ depends on $\sigma$.
  \end{corollary}
  \begin{proof}
    Let $(\lambda_j, v_j)_{j=0}^{\infty}$ denote the eigenvalues and eigenfunctions of the self-adjoint operator $\LL$.  Following \cite{blp17},  plugging the eigen-decomposition of a function $u$
    into the Riesz-Dunford calculus,  we can write the exact function $g(\LL) u$ as
    \begin{align*}
    g(\LL) u&=\sum_{j=0}^{\infty}{
              \Big(\frac{1}{2\pi \ii}\int_{\CC}{g(\psits(y)) (\psits-\lambda_j)^{-1} \psits'(y)\big(u,v_j\big)_{L^2(\Omega)} \,dy} \Big)\;v_j} 
    \end{align*}
    and analogously for the discrete approximation $Q^{\LL}(g,\Nquad)u$.
    For the norm, as defined in~\eqref{eq:def_HH}, this means:
    \begin{align*}
      \norm{E^{\LL}(g,\Nquad)}^2_{\mathbb{H}^{2r}(\Omega)}
      &= \frac{1}{4\pi^2}
        \sum_{j=0}^{\infty}{\Big|(1+\lambda_j^{r}) E^{\lambda_j}(g,\Nquad)\big(u,v_j\big)_{L^2(\Omega)}\Big|^2}  \\
      &\lesssim \sup_{\lambda \geq \lambda_0} \big|(1+\lambda^{r-s}) E^{\lambda}(g,\Nquad)\big|^2
        \norm{u}^2_{\mathbb{H}^{2s}(\Omega)}.
    \end{align*}
    We have thus reduced the problem to one of scalar quadrature, for which we
    aim to apply Lemma~\ref{lemma:sinc_in_dexp}.
    We fix $\lambda > \lambda_0>\kappa$.
    $\psits$ maps $\Dexp$ analytically to $\mathcal{O}$ via Lemma~\ref{lemma:psi:properties:right_halfplane}
    ($\delta$ depends on $\theta$ and $\omega$).
    What remains to be  shown is a pointwise bound for the
    function
    \begin{align*}
      h_\lambda(y):=\lambda^{r-s}g(\psits(y)(\psits-\lambda)^{-1} \psits'(y) \qquad \forall y \in \Dexp.
    \end{align*}
    By distinguishing the cases $\abs{\psits(y)}<\lambda/2$ and $\abs{\psits(y)}\geq \lambda/2$ we
    get using either  \eqref{eq:bound_dpsi} or Lemma~\ref{lemma:psi:properties:right_halfplane}
    \begin{align*}
      \lambda \abs{\psits(y)-\lambda}^{-1} \abs{\psits'(y)} \lesssim
      \abs{\psits(y)} \cosh(\Re(y)).
    \end{align*}
    We conclude using Lemma~\ref{lemma:psi:properties:right_halfplane}:
    \begin{align*}
    \abs{h_\lambda(y)}&\leq C_g \abs{\psits(y)}^{-\mu}
    \Big(\lambda \abs{\psits(y)-\lambda}^{-1} \abs{\psits'(y)}\Big)^{r-s}
    \Big(\abs{\psits(y)-\lambda}^{-1} \abs{\psits'(y)}\Big)^{1-r+s}\\
    &\lesssim  C_g \abs{\psits(y)}^{-\mu +r - s} \cosh(\Re(y)).
    \end{align*}
    The double exponential growth of $\psits$ (see Lemma~\ref{lemma:psi_growth}) and
    Lemma~\ref{lemma:sinc_in_dexp} then give (after
    absorbing the $\cosh$ term by slightly adjusting $\varepsilon$):
    \begin{align*}
      \lambda^{r-s}\big|{E^{\lambda}(g)}\big|
      &=\frac{1}{2\pi}\Big|
        \int_{\R}{h_{\lambda}(y)\,dy}
        -  k\sum_{n=-\infty}^{\infty}{h_\lambda(k\,n)}
        \Big|
        \lesssim C_g e^{-\gamma_1 \frac{\sqrt{\mu - r+s -2\varepsilon}}{\sqrt{k}}}.
    \end{align*}    
    The cutoff error is handled easily, also using this estimate.
    We calculate
    \begin{align*}
      \abs{Q^{\lambda}(g)-Q^{\lambda}(g,\Nquad)}
      &\leq k \!\!\! \sum_{ \abs{j}\geq \Nquad +1}{\!\!\!\abs{ h_\lambda(\psits(j\,k))}}
        \leq C_g \,k \lambda^{-r+s} \!\!\! \sum_{ \abs{j}\geq \Nquad +1}{\!\!\!\exp\big(-\gamma(\mu-r+s) e^{jk}\big)}  \\
      &\lesssim C_g \lambda^{-r+s}  \exp\big(-\gamma (\mu-r+s) e^{ \Nquad k}\big),
    \end{align*}
    where the last step follows by estimating the sum by the integral and elementary estimates.
  \end{proof}

  \section{The elliptic problem}
  \label{sect:elliptic}
In this section, we consider approximation of the elliptic fractional diffusion problem
\begin{align}
  \label{eq:elliptic_continuous_problem}
  \text{Given $f \in L^2(\Omega)$ and $\beta >0$}, \quad
  \text{find $u \in \dom(\LL^\beta)$ such that } \quad \LL^{\beta} u&=f.
\end{align}
Using the Riesz-Dunford formula, this is equivalent to computing
\begin{align*}
u&=\frac{1}{2\pi \ii}
\int_{\CC}{z^{-\beta}\big(\LL-z\big)^{-1} f\,dz}.
\end{align*}

In order to get a (semi-) discrete scheme, we replace the integral with
the quadrature formula.
Given $\Nquad \in \N$ and $k >0$,
the fully discrete approximation to~\eqref{eq:elliptic_continuous_problem} is then given by
 $u_k:=Q(z^{-\beta},\Nquad) f$.

\begin{remark}
For a fully discrete scheme, one would in addition replace $(\LL-z)^{-1}$ by a Galerkin solver.
Given an closed subspace $\HHx \subseteq H^1(\Omega)$, the discrete resolvent $R_h(z): L^2(\Omega)\to \HHx $ is given as the solution
\begin{align*}
   R_h(z) f:=u_h, \quad \text{with} \quad
  \big(\mathfrak{A} \nabla u_h,\nabla v_h)_{L^2(\Omega)} + \big((\mathfrak{c}-z)u_h,v_h\big)_{L^2(\Omega)} &= \big(f,v_{h}\big)_{L^2(\Omega)} \quad \forall v_{h} \in \HHx.
\end{align*}
Given discretization parameters $\HHx \subseteq \mathbb{H}^{1}(\Omega)$, $\Nquad \in \N$ and $k >0$,
the fully discrete approximation to~\eqref{eq:elliptic_continuous_problem} is then given by
\begin{align}
  \label{eq:elliptic_discretization}
u_{h,k}:= 
\frac{1}{2\pi \ii} \sum_{j=-\Nquad}^{\Nquad}
{
  \big(\psits(jk)\big)^{-\beta} \psits'(jk)  \,R_h\big(\psits(jk)\big)f.
}
\end{align}
In order to keep presentation to a reasonable length, we focus on the spatially continuous setting.
We only remark that
discretization in space can be easily incorporated into the analysis. For low order finite elements one can follow~\cite{blp17}; for an exponentially convergent
$hp$-FEM scheme we refer to~\cite{MR20_hp_sinc}.
\end{remark}

\begin{remark}
  \label{rem:complex_arithmetic}
  We should point out that for the elliptic problem,
  there exist methods based on the Balakrishnan formula
  (see also Section~\ref{sect:numerics}) which do not require complex arithmetic.
  On the other hand, since we are only approximating real valued functions, we can exploit the
  symmetry of \eqref{eq:def_quadrature_operator} to only  solve for $j \geq 0$, thus halving the number of
  linear systems. This results in (roughly) comparable computational effort for both the
  Balakrishnan and the double exponential schemes. Due to their better convergence the DE-schemes might therefore
  still be advantageous.
\end{remark}

\subsection{Error analysis}
In order to analyze the quadrature error, we need to understand a specific
scalar function. This is done in the next Lemma.
\begin{lemma}
  \label{lem:analyze_power_function}
  Fix $\lambda > \lambda_0 > \kappa$ and $\beta > 0$. For
  $y \in \R$, define the function
  \begin{align*}
    g_\lambda^\beta(y)
    &:=\big(\psits(y)\big)^{-\beta} \Big(\psits(y)-\lambda\Big)^{-1} \psits'(y).
  \end{align*}
  Then the following statements hold:
  \begin{enumerate}[(i)]
  \item 
    \label{item:g_l_b_meromorphic_continuation}
    $g_\lambda^{\beta}$ can be
    extended to a meromorphic function on $\Dt$. 
    It  has finitely many poles.
    They satisfy $\psits(y)=\lambda$ and are all simple.
    For any  $\nu\geq 0$, the number of poles $p$ within the strip
    $$
    \nu-\frac{1}{\ln(\lambda)}\leq\abs{ \Im(p)} \leq \nu+\frac{1}{\ln(\lambda)}
    $$
    can be bounded independently of $\nu$, $\beta$ and $\lambda$.
    The imaginary part of $p$ can be bounded away from zero 
    and for large $\lambda$, the following asymptotics hold:
    \begin{align}
      \label{eq:g_l:pos_of_poles}
    \abs{\Im(p)}
      &\geq
        \begin{cases}
          \frac{\atan(\theta)}{\ln(\lambda/\kappa)}   - \bigO\Big(\ln(\lambda/\kappa)^2\Big) & \text{if $\sigma=1$}, \\
          \frac{\pi}{2\ln(\lambda/\kappa)} - \bigO\Big(\ln(\lambda/\kappa)^2\Big) &\text{if $\sigma=1/2$}.
        \end{cases}        
    \end{align}
    where  the implied constants  depend on $\theta$, $\kappa$, and $\lambda_0$.
  \item
    \label{item:g_l_b_pointwise_bounds}
    There exist constants $C>0$, $\gamma > 0$, independent of $\lambda$ and $\beta$
    and a value $d_{\lambda} \in (d(\theta)/2,d(\theta))$ 
    such that
    $g_\lambda^\beta$ satisfies the bounds
    \begin{align}
      \big|{(1+\lambda^{\frac{\beta}{2}})g_\lambda^\beta(a \pm \ii \,d_{\lambda} )}\big|
      \leq C  \exp\big(- \gamma \beta  e^a\big) \qquad \forall a \in \R.
    \end{align}
  \item
    \label{item:g_l_b_integral_bounds}
    There exists a constant $C>0$ such that for $d_{\lambda}$ from~(\ref{item:g_l_b_pointwise_bounds})
    and $\beta \geq \overline{\beta}>0$
    \begin{align*}
      \int_{\R} { (1+\lambda^{\frac{\beta}{2}})\big|{g_\lambda^\beta(w\pm\ii d_\lambda)}\big| \, dw}
      &\leq C < \infty.        
    \end{align*}
    The constant $C$ may depend on $\overline{\beta}$
    but can be chosen independently of $\lambda$ and $\beta$.
  \end{enumerate}
\end{lemma}
\begin{proof}
  Ad~(\ref{item:g_l_b_meromorphic_continuation}):  
  We note that by Lemma~\ref{lemma:psi_properties:analytic}, $\psits$ is non-vanishing in $\Dt$.
  Since $\Dt$ is simply connected, we may define
  $$
  h(y):=1+\int_0^y{ \frac{\psits'(\zeta)}{\psits(\zeta)} \,d\zeta}.
  $$
  It is easy to check that on $\R$ we have $h(y)=\ln(\psits(y))$
  since the derivative as well as the value at $y=0$ coincide.
  Thus, defining
  $$
  g_\lambda^\beta(y)
  :=e^{-\beta h(y)} \big(\psits(y)-\lambda\big)^{-1} \psits'(y)
  $$
  provides a valid meromorphic extension.
  The only poles are located where $\psits(z)=\lambda$.
  By Lemma~\ref{lemma:psi_hits_lambda}(\ref{it:psi_hits_lambda:finite_number}),
  the number of such poles within strips of width $\ln(\lambda)^{-1}$ is uniformly bounded.
  Since, by Lemma~\ref{lemma:psi_properties:analytic},  
  $\psits'$  has no zeros in the domain $\Dt$  all the poles are simple.
  The bound on the imaginary part follows from
  Lemma~\ref{lemma:psi_hits_lambda}(\ref{it:psi_hits_lambda:scale_of_imag_part}).

  Ad~(\ref{item:g_l_b_pointwise_bounds}):
  We first note for $y=a \pm d_{\lambda}$ if
  $\lambda < \abs{\psits(y)}/2$ the trivial estimate
  $ \abs{\psits(y) - \lambda}^{-1}
  \leq\frac{2}{\abs{\psi(y)}}$ holds.
  Otherwise, we use  Lemma~\ref{lemma:psi_hits_lambda}(\ref{it:psi_hits_lambda:viable_path})
  to get
  \begin{align*}
    \abs{\psits(y) - \lambda}^{-1}
    \lesssim \lambda^{-1} \leq 2\abs{ \psits(y)}^{-1}.
  \end{align*}
  Overall, we can estimate using Lemma~\ref{lemma:psi_growth}
  \begin{align*} 
    |{g^\beta_\lambda(y)}|
    &\lesssim
      \abs{\psits(y)}^{-\beta} \abs{\psits(y)-\lambda}^{-1} \abs{\psits'(y)}
      \lesssim \abs{\psits(y)}^{-\beta-1} \abs{\psits'(y)}
      \lesssim \exp(-\gamma e^{\Re(y)}),
  \end{align*}
  where in the last step, we used that $\psits'$ has the same asymptotic behavior
  as $\psits$ up to single exponential terms, which we absorb into the double exponential by
  slightly reducing $\gamma$.
  
  Looking at $|{\lambda^{\beta/2} g^\beta_\lambda(y)}|$, one can calculate
  using two different ways to estimate $\psits(y)-\lambda$:
  \begin{align*} 
    \lambda^{\beta/2}|{g^\beta_\lambda(y)}|
    &\lesssim
      \abs{\psits(y)}^{-\beta}
      \Big(\underbrace{\lambda \abs{\psits(y)-\lambda}^{-1}}_{\lesssim 1}\Big)^{\beta/2}
      \Big(\underbrace{\abs{\psits(y)-\lambda}^{-1}}_{\lesssim \abs{\psits(y)}^{-1}}\Big)^{1-\beta/2}\abs{\psits'(y)} \\
    &\lesssim \abs{\psits(y)}^{-\beta}
       \abs{\psits(y)}^{-1+\beta/2}
      \abs{\psits'(y)}      
    \stackrel{Lemma~\ref{lemma:psi_growth}}{\lesssim} \exp\Big(-\frac{\gamma \beta}{2} e^{\Re(y)}\Big).
  \end{align*}
  The integral bound then follows easily from the pointwise ones.
\end{proof}

\begin{theorem}[Double Exponential formulas for elliptic problems]
  \label{thm:de_for_elliptic}
  Fix $\lambda_0 > \kappa$, $\overline{\beta}>0$ and $r\in [0,\beta/2]$.
  Then
  there exist constants  $C>0$, $\gamma >0, \gamma_1>0$ such that
  for $\lambda > \lambda_0$, $\beta \geq \overline{\beta}$, $k>0$, $\Nquad \in \N$, the following estimate holds
  \begin{subequations}
  \begin{align}
    \label{eq:de_for_elliptic:1}
    \lambda^{r}\abs{E^{\lambda}(z^{-\beta},\Nquad)}
    &\lesssim k^2\max(1,\ln(\lambda))^2 \lambda^{-\beta+r}
      e^{-\frac{\max\{p(\sigma,\theta,\lambda),\gamma_1\}}{k\max(1,\ln(\lambda/\kappa))}}
      + 
      e^{-\frac{\gamma}{k}} + \exp(- \gamma \beta e^{k \Nquad}),
  \end{align}
  where the rate is given by
  \begin{align}
    \label{eq:de_elliptic_rate}
    p(\sigma,\theta,\lambda)= \begin{cases}
      2\pi \atan(\theta)- \frac{c_2}{\ln(\lambda)}  & \text{if  $\sigma=1$,} \\
      \pi^2 - \frac{c_2}{\ln(\lambda)} & \text{if  $\sigma=1/2$.} 
    \end{cases}
  \end{align}
\end{subequations}
Thus for $k \sim \ln(\Nquad)/\Nquad$ we get (almost) exponential  convergence:
  \begin{align}
    \label{eq:de_for_elliptic:2}
    \lambda^{r}\Big|E^{\lambda}(z^{-\beta},\Nquad)\Big|
    &\lesssim
      k^2 \max(1,\ln(\lambda/\kappa))^2 \lambda^{-\beta+r} e^{-\frac{\max\big\{p(\sigma,\theta,\lambda),\gamma_1\big\}\Nquad}{\ln(\lambda/\kappa)\ln(\Nquad)}} +
      e^{-\gamma' \frac{\Nquad}{\ln(\Nquad)}}.
  \end{align}
  The implied constants and $\gamma$
  may depend on  $\lambda_0$, $\overline{\beta}$, $\sigma$, $\theta$ and $\kappa$. 
\end{theorem}
\begin{proof}
  To cut down on notation, we only consider the case $\ln(\lambda/\kappa)\geq c_1 >1$
  so that the first term in the minimum of ~\eqref{eq:g_l:pos_of_poles}  dominates.
  If $\lambda$ is small, the error can be absorbed into the $e^{-\gamma/k}$ term.
  The error $E^{\lambda}(z^{-\beta},\Nquad)$ corresponds to approximating
  $g_{\lambda}^{\beta}$ by $\sinc$ quadrature.
  We split the error into two parts, the quadrature error and the cutoff error. 
  \begin{align*}
    \lambda^r\Big|{\int_{-\infty}^{\infty}{g_{\lambda}^{\beta}(y) \, dy}
    - k \sum_{ j= -\Nquad}^{\Nquad} g_\lambda^\beta(j\,k)}\Big|
  \leq
    \underbrace{\lambda^r\Big|{\int_{-\infty}^{\infty}{g_{\lambda}^{\beta}(y)\,dy}
    - k \sum_{ j= -\infty}^{\infty} g_{\lambda}^{\beta}(j\,k)}\Big|}_{=E^{\lambda}(z^{-\beta})}
    + \underbrace{k \lambda^r  \!\!\! \sum_{ \abs{j} > \Nquad =1} \abs{ g_{\lambda}^{\beta}(j\,k)}}_{
    =:E_c}.
  \end{align*}
  The term $E_c$ can be handled by the same argument as in
  Corollary~\ref{cor:sinc_in_dexp_for_quadrature_operators}.
  We therefore focus on the quadrature error $E^{\lambda}(z^{-\beta})$
  and apply Proposition~\ref{prop:sinc_with_poles}.
  By Lemma~\ref{lem:analyze_power_function}(\ref{item:g_l_b_integral_bounds})
  it holds that $N\big(g_\lambda^\beta,\Dt\big) < \infty$.
  To satisfy assumption~(\ref{def:sinc_functions:ii}),
  it suffices that  (for sufficiently large $y$) the vertical strips do not contain any poles
  and we can use the asymptotics of Lemma~\ref{lem:analyze_power_function}(\ref{item:g_l_b_pointwise_bounds}).
  
  By Lemma~\ref{lem:analyze_power_function}, there are at most finitely many simple poles.
  The reside of the function at these poles can be easily calculated
  using the well-known rule
  $$
  \operatorname{res}\big(f/g: z_0\big)
  = \frac{f(z_0)}{g'(z_0)},
  $$
  provided that $f$ is analytic and $g'(z_0)\neq 0$.
  In our case this means, if $\psits(y_\lambda)=\lambda$:
  \begin{align*}
  \operatorname{res}(g^{\beta}_\lambda; y_\lambda)
    &=\frac{(\psi(y_\lambda))^{-\beta} \psi'(y_\lambda)}{\psi'(y_\lambda)}
      = (\psi(y_\lambda))^{-\beta} = \lambda^{-\beta}.
  \end{align*}
  Thus, for a single pole $y_\lambda$ with $s_{y_\lambda}:=\operatorname{sign}(\Im(y_\lambda)) $ we can estimate
  \begin{align*}
    \abs{e^{\ii  \frac{\pi s_{y_\lambda} y_\lambda}{k}} \, \operatorname{res}(g^\beta_\lambda,y_\lambda) \, \gamma(k;y_\lambda)}
    &\lesssim
      \lambda^{-\beta} \frac{e^{-2\abs{\Im(y_\lambda)}/k}}{1-e^{-2\abs{\Im(y_\lambda)}/k}}.
  \end{align*}
  By Lemma~\ref{lem:analyze_power_function}(\ref{item:g_l_b_meromorphic_continuation}),
  we can group poles into buckets of size $\frac{1}{\ln(\lambda/\kappa)}$, denoted by
  $$
  B_{\ell}:=\Bigg\{ y: \psits(y) = \lambda
  \quad \text{with} \quad
  \frac{\frac{p(\sigma,\theta,\lambda)}{2\pi}+\ell}{\ln(\lambda/\kappa)}
  \leq \abs{\Im(y)} \leq
  \min\Big(\frac{\frac{p(\sigma,\theta,\lambda)}{2\pi}+\ell+1}{\ln(\lambda/\kappa)},d(\theta)\Big) \Bigg\}
  $$
  such that the number of elements in each bucket $B_{\ell}$ is uniformly bounded (independently of $\lambda$, $\beta$ and $\ell$).
  This allows us to calculate for the pole contribution in Proposition~\ref{prop:sinc_with_poles}:
  \begin{align*}
    \Big|\pi \sum_{y_\lambda \in P^y_\lambda} {e^{\ii  \frac{s_\lambda \pi y_\lambda}{k}} \,
    \operatorname{res}(g_\lambda^{\beta};y_\lambda) \, \gamma(k;y_\lambda)}\Big|
    &\leq \lambda^{-\beta} \pi
      \sum_{\ell=0}^{\infty}
      {
      \Big|
      \sum_{y_\lambda \in B_{\ell}} {
      e^{\ii  \frac{ \pi s_{y_\lambda}y_\lambda}{k}} \,  \gamma(k;y_\lambda)} \Big|}       
    \lesssim
      \lambda^{-\beta} 
      \sum_{\ell=0}^{\infty}
      { \frac{e^{-\frac{p(\sigma,\theta,\lambda)+\ell}{k\ln(\lambda/\kappa)}}}{1-e^{-\frac{p(\sigma,\theta,\lambda)+\ell}{k\ln(\lambda/\kappa)}}}} \\
    &\lesssim
      \frac{\lambda^{-\beta} }{1-e^{-\frac{p(\sigma,\theta,\lambda)}{k\ln(\lambda/\kappa)}}}
      \sum_{\ell=1}^{\infty}
      { e^{-\frac{p(\sigma,\theta,\lambda)+\ell}{k\ln(\lambda/\kappa)}}}
      \lesssim 
      \lambda^{-\beta} \ln(\lambda)^2 k^2
      e^{-\frac{p(\sigma,\theta,\lambda)}{k\ln(\lambda/\kappa)}},
  \end{align*}
  where we used the elementary estimate $1-e^{-2x}\gtrsim \min(x,1)$ for $x \geq 0$.
  
  Applying Proposition~\ref{prop:sinc_with_poles} and inserting this estimate for the pole-contributions gives:
  \begin{align*}
    \lambda^{r} E^{\lambda}(z^{-\beta})
    &=E^{\lambda}(\lambda^{r} z^{-\beta}) 
    \stackrel{{Prop.~\ref{prop:sinc_with_poles}}}{\lesssim}
       \frac{e^{-2\pi d_\lambda/k}}{1-e^{-2\pi d_\lambda/k}} N(g_{\lambda}^{\beta},\mathcal{D}_{d_\lambda})
      +      \lambda^{-\beta+r}  \ln(\lambda)^2 k^2
      e^{-\frac{p(\sigma,\theta,\lambda)}{k\ln(\lambda/\kappa)}}. \qedhere
  \end{align*}
\end{proof}

The previous estimate gives (almost) exponential convergence with respect to $\Nquad$. But the rate of the exponential
deteriorates like $1/\ln(\lambda)$ for large $\lambda$. In the following corollary, we give a $\lambda$-robust version of this
estimate. We allow for an additional factor $\lambda^{\rho}$ which will allow us to make use of possible additional smoothness when
considering function-valued integrals.

\begin{corollary}
  \label{cor:de_for_elliptic:lambda_robust}
  Fix $\lambda_0 > \kappa >0$, $\overline{\beta}>0$ and $r \in [0,\beta/2]$. Then,
  for every $\varepsilon \geq 0$,
  there exist constants  $C>0$, $\gamma >0$ such that
  for $\lambda > \lambda_0$, $\beta>\overline{\beta}$,  $\rho\geq 0$, $k>0$, $\Nquad \in \N$,
  the following estimate holds
  \begin{subequations}
  \begin{align}
    \label{eq:de_for_elliptic:lambda_robust}
    \lambda^{r}\Big|E^{\lambda}(z^{-\beta},\Nquad)\Big|
    &\lesssim  \exp\Big(-\frac{[p(\sigma,\theta)-\varepsilon] \sqrt{\beta+\rho-r}}{\sqrt{k}}\Big) \lambda^{\rho}+ 
      \exp(-\frac{\gamma}{k}) + \exp(- \gamma e^{k \Nquad}).
  \end{align}
  where the rate $p(\sigma,\theta)$ is given by
  \begin{align}
    \label{eq:1}
    p(\sigma,\theta):=
    \begin{cases}
      2 \sqrt{2\pi\atan(\theta)}  & \text{for }\sigma=1, \\
      2\pi, & \text{for }\sigma=1/2. 
    \end{cases}
  \end{align}
\end{subequations}
For $\varepsilon >0$, the implied constant in the estimate and $\gamma$ may depend on $\lambda_0$, $\sigma$, $\theta$,
$\overline{\beta}$, $\kappa$.  If $\varepsilon=0$, the constants in addition depend on $\rho$ and $\beta$.
\end{corollary}
\begin{proof}
  We first show the estimate for $\varepsilon >0$.
  We note that for $\ln(\lambda/\kappa) \geq k^{-1}$,  we can bound the error in Theorem~\ref{thm:de_for_elliptic} by $\exp(-\gamma/k)$  (for an appropriate choice of constant $\gamma$)
  due to the smallness of the term $\lambda^{-\beta}$.
  Thus it remains to consider the case $\ln(\lambda/\kappa) < k^{-1}$. Similarly, if
  $\ln(\lambda )\leq \max(\frac{c_2}{\varepsilon},-\ln(\kappa)\frac{p(\sigma,\theta)-2\varepsilon}{\varepsilon},1)=:\mu_0$, the
  leading error term behaves like $\exp(-\gamma\frac{\mu_0}{k})$. We are left to consider the remaining case.
  Writing $\mu:=\ln(\lambda)$, the error term can be estimated: 
  \begin{align}
    k^2 \ln(\lambda/\kappa)^2\lambda^{-\beta+r-\rho} e^{-\frac{p(\sigma,\theta,\lambda)}{\max(1,\ln(\lambda/\kappa)) k}} \, \lambda^{\rho}
    &\lesssim
      \exp\Bigg(-(\beta+\rho-r) \mu - \frac{p(\sigma,\theta) - \frac{c_2}{\mu}}{(\mu-\ln(\kappa)) k}\Bigg)  \nonumber \\ 
    &\lesssim
          \exp\Bigg(-(\beta+\rho-r) \mu - \frac{p(\sigma,\theta) - 2\varepsilon}{\mu k}\Bigg).
      \label{eq:cor:de_for_elliptic_int1}
  \end{align}
  We look for the minimum of the exponent. Setting the derivative of the map
  $$
  \mu \mapsto -(\beta+\rho-r) \mu - \frac{p(\sigma,\theta)- 2\varepsilon}{\mu k}
  $$
  to zero, we get that the minimum satisfies
  $$
  0=-(\beta+\rho-r) \mu_{\min}^2 + \frac{p(\sigma,\theta) - 2\varepsilon}{k},
  \qquad \text{or} \qquad
  \mu_{\min}:=\sqrt{ \frac{1}{(\beta+\rho-r)} \frac{p(\sigma,\theta) - 2\varepsilon}{k}}.
  $$
  Inserting this value into \eqref{eq:cor:de_for_elliptic_int1} gives the stated result
  (after slightly changing $\varepsilon$ to get to the stated form).

To see the case for $\delta=0$, we note that if $\ln(\lambda/\kappa) \leq \frac{\gamma_1 k^{-1/2}}{p(\sigma,\theta)\sqrt{\beta+\rho-r}}$,
we can estimate for the leading term in Theorem~\ref{thm:de_for_elliptic}:
\begin{align*}
    k^2 \ln(\lambda/\kappa)^2 \lambda^{-\beta+r-\rho} e^{-\frac{\gamma_1}{\mu k}} \, \lambda^{\rho}
  &\leq
    k^2 \ln(\lambda/\kappa)^2 \lambda^{-\beta+r-\rho} e^{-\frac{p(\sigma,\theta)\sqrt{\beta+\rho-r}}{\sqrt{k}}} \, \lambda^{\rho}.
\end{align*}
In the remaining case, we can estimate the higher order term in the  $\mu$-asymptotics as
$$
e^{\frac{c_2}{\mu^2 k}}\leq e^{\frac{c_2 p(\sigma,\theta) \sqrt{\beta+\rho-r}}{\gamma_1}}
=: C(\sigma,\theta,\beta,\rho).
$$
We can also estimate
$$\lambda^{-\beta+r-\rho}\leq \kappa^{-\beta+r-\rho}\Big( \frac{\lambda}{\kappa}\Big)^{-\beta+r-\rho}$$
and continue as in the proof for $\delta>0$ but using $\mu:=\ln(\lambda/\kappa)$.
This time we no longer have to compensate for the factors involving $c_2/\mu$ and $-\ln(\kappa)$  by
slightly reducing the rate. The price we pay is that the constant may blow up for $\rho\to \infty$.
\end{proof}

We can now leverage our knowledge about the function $g^\beta_\lambda$ to gain insight into the discretization error for~\eqref{eq:elliptic_discretization}.

\begin{corollary}
  \label{cor:elliptic_full_discretization}
  Let $u$ be the exact solution to~\eqref{eq:elliptic_continuous_problem} and assume
  $f \in \mathbb{H}^{2\rho}(\Omega)$ for some $\rho \geq 0$. Let $\beta \geq \overline{\beta}>0$
  and $u_k:=Q^{\LL}(z^{-\beta},\Nquad) f$
  denote the approximation computed using stepsize $k>0$
  and $\Nquad \in \N$ quadrature points.
  Then, the following estimate holds for all $\varepsilon \geq 0$ and $r\in [0,\beta/2]$:
  \begin{align*}
    \norm{u-u_{k}}_{\mathbb{H}^{2r}(\Omega)}
    &= \norm{E^{\LL}(z^{-\beta},\Nquad)f}_{\mathbb{H}^{2r}(\Omega)} \\
    &\lesssim
      e^{-\frac{[p(\sigma,\theta)-\varepsilon] \sqrt{\beta+\rho-r}}{\sqrt{k}}}\norm{f}_{\mathbb{H}^{2\rho}(\Omega)}     
      \!+\! \Big[\exp(-\frac{\gamma}{k})  + \exp(- \gamma e^{k \Nquad})\Big] \norm{f}_{L^2(\Omega)}.
  \end{align*}
  For $\varepsilon >0$,
  the implied constant and $\gamma$ may depend on $\varepsilon, r$, the smallest eigenvalue $\lambda_0$ of $\LL$, $\overline{\beta}$, $\kappa$, $\theta$ and $\sigma$. But they are independent of $\rho$, $\beta$,
  $k$, and $f$.  If $\varepsilon=0$, the constants may in addition depend on $\rho$ and $\beta$.
\end{corollary}
\begin{proof}
  Let $(\lambda_j, v_j)_{j=0}^{\infty}$ denote the eigenvalues and eigenfunctions of the self-adjoint operator $\LL$. Just as we did in the proof of Corollary~\ref{cor:sinc_in_dexp_for_quadrature_operators},
  we plug the eigen-decomposition into the Riesz-Dunford calculus and
  Definition~\ref{def:quadrature_operators} to get for the discretization error:
  \begin{align*}
    \norm{u-u_{k}}_{\mathbb{H}^{2r}(\Omega)}^2
    &=
      \sum_{j=0}^{\infty}{
      \Big|(1+\lambda_j^{r})\frac{1}{2\pi \ii}\int_{\CC}{g^\beta_{\lambda_j}(y)  \,dy} - \frac{1}{2\pi \ii}\sum_{n=-\Nquad}^{\Nquad}{g^\beta_{\lambda_j}(k\,n)}\Big|^2  \abs{ \big(f,v_j\big)_{L^2(\Omega)}}^2}.
  \end{align*}
  Applying Corollary~\ref{cor:de_for_elliptic:lambda_robust} then gives for $\rho \geq 0$
  \begin{align*}
    \norm{u-u_{k}}_{\mathbb{H}^{2r}(\Omega)}^2     
      &\lesssim
        e^{-\frac{2[p(\sigma,\theta)-\varepsilon] \sqrt{\beta+\rho-r}}{\sqrt{k}}} \sum_{j=0}^\infty{\lambda_j^{2\rho} \abs{ \big(f,v_j\big)_{L^2(\Omega)}}^2} 
      +\big[e^{-\frac{\gamma}{k}} + e^{- \gamma e^{k \Nquad}} \big]\norm{f}^2_{L^2(\Omega)} \\
      &\lesssim
        e^{-\frac{2[p(\sigma,\theta)-\varepsilon] \sqrt{\beta+\rho-r}}{\sqrt{k}}}  \norm{f}^2_{\mathbb{H}^{2\rho}(\Omega)} 
        +\Big[e^{-\frac{\gamma}{k}} + e^{- \gamma e^{k \Nquad}} \Big] \norm{f}^2_{L^2(\Omega)}.
        \qedhere
  \end{align*}
\end{proof}

\begin{remark}
  When comparing Corollary~\ref{cor:elliptic_full_discretization} to the estimates of the
  standard $\sinc$-quadrature one might think that the double exponential method is inferior due to the $\sqrt{\kquad}$
  vs $\kquad$ behavior. This misconception can be cleared up by considering the better decay properties of the
  double-exponential formula. It allows to 
  choose $\kquad \sim \ln(\Nquad)/\Nquad$ compared to the standard $\sinc$-quadrature choice of $\kquad \sim \Nquad^{-1/2}$
  without the cutoff error becoming dominant.
\end{remark}
\begin{remark}
  For most of the computation, the convergence rate is determined by the factor $p(\sigma,\theta)$ in Corollary~\ref{cor:de_for_elliptic:lambda_robust}. We observe
  that for $\theta=1$, picking $\sigma=1/2$ roughly doubles the convergence rate. Similarly, it often appears beneficial to pick larger values of $\theta$.
  Especially for $\sigma=1$, we get an asymptotic rate for $\theta \to \infty$, which the same to the case of $\sigma=1/2$. But we need to point out that
  increasing $\theta$ means that we have to decrease the value $d(\theta)$, which determines the rate in the higher orders terms of the form $e^{-\gamma/k}$, thus leading to
  those terms dominating in a larger and larger preasymptotic regime. Overall, the method using $\sigma=1/2$ and setting $\theta$ moderately large is expected to give the best convergence rates;
  cf. Section~\ref{sect:numerics}.
\end{remark}

The previous corollary shows that in general, the convergence behaves like $\bigO(e^{-\frac{\gamma}{\sqrt{k}}})$. It also shows that, if the
function $f$ in the right-hand side has some additional smoothness, the method automatically detects this and delivers an improved
convergence rate. For an even smaller subset of possible right-hand sides,
namely those reflected by a Gevery-type class with boundary conditions, this effect can even lead to an improved
convergence of the form $\bigO(e^{-\frac{\gamma}{{k\abs{\ln(k)}}}})$.
Examples for such functions are those only containing a finite number of frequencies when decomposed into the eigenbasis of $\LL$, but also more complex functions such as smooth bump functions with compact support are admissible
(see \cite[Section~1.4]{Ro93}).
The details can be found in the following corollary:
\begin{corollary}
  \label{cor:elliptic_full_discretization:super_smooth}
  Let $u$ be the exact solution to~\eqref{eq:elliptic_continuous_problem} and assume
  that  there exist constants $C_{f}, \omega, R_f>0$ such that
  \begin{align*}
    \norm{f}_{\mathbb{H}^{\rho}(\Omega)} &\leq C_{f} \, R_f^\rho \,\big(\Gamma(\rho+1)\big)^{\omega}  < \infty \qquad  \forall \rho\geq 0.
  \end{align*}
  Assume that $\beta>\overline{\beta}>0$.  
  Let $u_k:=Q^{\LL}(z^{-\beta},\Nquad)f$ denote the approximation computed using stepsize $k \in(0,1/2)$
  and $\Nquad \in \N$ quadrature points.
  Then, the following estimate holds:
  \begin{align*}
    \norm{u-u_{k}}_{\mathbb{H}^{\beta}(\Omega)}
    =\norm{E^{\LL}(z^{-\beta},\Nquad)}_{\mathbb{H}^{\beta}(\Omega)}
    &\lesssim
      C_f \exp\Big(-\frac{\gamma}{k\abs{\ln(k)}}\Big) + C_f\exp\Big(- \gamma e^{k \Nquad}\Big) .
  \end{align*}
  The implied constant and $\gamma$ may depend on $\omega$, the smallest eigenvalue $\lambda_0$ of $\LL$,  $\kappa$, $\theta$, $\sigma$, $R_f$, $\overline{\beta}$,  and $\omega$.
  If $\omega=0$, the logarithmic term may be removed.
\end{corollary}
\begin{proof}
  For simplicity of notation, we ignore the cutoff error, i.e., for now consider $\Nquad=\infty$.
  The cutoff error can either be easily tracked throughout the proof or added at the
  end, analogously to Corollary~\ref{cor:sinc_in_dexp_for_quadrature_operators}.

  We first note, that by Stirling's formula, we can estimate the derivatives of $f$ by
  \begin{align*}
    \norm{f}_{\mathbb{H}^{\rho}(\Omega)}   &\leq \widetilde{C}_{f} \exp\big( \rho(\omega \ln(\rho)+ c_2))\big).
  \end{align*}
  By assumption, we can apply Corollary~\ref{cor:elliptic_full_discretization} for any $\rho\geq 0$. Picking
  $\rho=\frac{\delta}{k\ln(k)^2}$ for $\delta $  sufficiently small and $\varepsilon:=p(\sigma,\theta)/2$
  (because we need $\rho$-robust error estimates)
  gives:
  \begin{align*}
    \norm{u-u_{k}}_{\mathbb{H}^{\beta}(\Omega)}
    &\lesssim
      \exp\Big(-\frac{p(\sigma,\theta) \sqrt{\beta/2+\delta k^{-1}{\abs{\ln(k)}}^{-2}}}{2\sqrt{k}}\Big) \norm{f}_{\mathbb{H}^{\frac{2\delta}{k\ln(k)^2}}(\Omega)}
       +e^{-\frac{\gamma}{k}} \norm{f}_{L^2(\Omega)} \\
    &\lesssim
      \exp\Bigg(\!\!\!-\!\frac{\sqrt{\delta} \gamma'}{k\abs{\ln(k)}}\Bigg) C_{f} \exp\Bigg(\frac{2\delta}{k\abs{\ln(k)}^2}\Big(\omega \ln\Big(\frac{2\delta}{k\abs{\ln(k)}^2}\Big)+ c_2\Big)\!\!\Bigg)
      +e^{-\frac{\gamma}{k}} \norm{f}_{L^2(\Omega)}
    \\
    &\lesssim \exp\Bigg(\!\!\!-\!\frac{\sqrt{\delta} }{k\abs{\ln(k)}}\Big( \gamma'- \frac{2\sqrt{\delta}}{\abs{\ln(k)}}\Big(\omega \ln\Big(\frac{2\delta}{k\abs{\ln(k)}^2}\Big)+ c_2\Big)\!\!\Bigg)
      +e^{-\frac{\gamma}{k}} \norm{f}_{L^2(\Omega)}.
  \end{align*}
  Expanding the logrithmic expression, one can then see that
  for $\delta$ small enough (but independent of $k$), the second term in the exponent is smaller than $\gamma'$ and the statement follows.
  If $\omega=0$, we don't have to compensate the factor $e^{\omega \rho \ln(\rho)}$, therefore picking $\rho \sim   k^{-1}$ is sufficient and the improved statement follows.
\end{proof}

\section{The parabolic problem}
\label{sect:parabolic}
In this section, we consider a time dependent problem.
We fix $\alpha, \beta \in (0,1]$ and a  final time $T>0$.
Given an initial condition $u_0 \in L^2(\Omega)$
and right-hand side $f \in C([0,T],L^2(\Omega))$ we seek $u:[0,T]\to \mathbb{H}^{\beta}(\Omega)$
satisfying
\begin{align}
  \label{eq:parabolic_model_problem}
  \partial^{\alpha}_t u + \LL^{\beta} u &= f \;\, \text{in } \Omega \times [0,T],   \qquad
                                          u(t)|_{\partial \Omega}=0 \;\; \forall t>0  \qquad
                                          u(0)=u_0 \;\, \text{in } \Omega.
\end{align}
Where $\partial_t^{\alpha}$ denotes the Caputo fractional derivative.
Following~\cite{bonito_pasciak_parabolic}, the solution $u$ can be written
using the Mittag-Leffler function $e_{\alpha,\mu}$ (see \eqref{eq:mittag_leffler})
as
\begin{align}
  \label{eq:parabolic_representation_formula}
  u(t)&:= e_{\alpha,1}\big(-t^{\alpha} \LL^{\beta}\big) u_0 + \int_{0}^{t} { \tau^{\alpha-1} e_{\alpha,\alpha}\big(-\tau^{\alpha} \LL^\beta\big) f(t-\tau)\,d\tau}.
\end{align}
Here we again  use either the spectral or, equivalently, the Riesz-Dunford calculus to
define the operators.
We discretize this problem by using our double exponential formula. Namely
for $k >0 $ and using $\Nquad \in \N$ quadrature points,
\begin{align}
  \label{eq:2}
  u^k(t)
  &:=Q^{\LL}\Big(e_{\alpha,1}(-t^{\alpha} z^{\beta}),\Nquad\Big)
  + \int_{0}^{t}{\tau^{\alpha-1} Q^{\LL}\Big( e_{\alpha,\alpha}(-\tau^{\alpha} z^{\beta}), \Nquad\Big)\,d\tau}.
\end{align}
Note that in practice, one would again replace the resolvent by a Galerkin solver and the
convolution in time by an appropriate quadrature scheme.
We refer to~\cite{bonito_pasciak_parabolic} for a low order approximation scheme and
\cite{MR20_hp_sinc} for an exponentially convergent scheme based on $hp$-finite elements
and $hp$-quadrature. In order to not overwhelm the presentation of the paper, we again do not
consider these types of discretization errors. But the analysis of those can be taken almost
verbatim from the references.

\subsection{The Mittag Leffler function}
The representation~(\ref{eq:parabolic_representation_formula}) hints
that it is crucial to understand the Mittag-Leffler function if one
wants to analyze the time dependent problem~(\ref{eq:parabolic_model_problem}).
We follow~\cite[Section 1.8]{ksh06}.
For parameters $\alpha>0$, $\mu \in \R$, the Mittag-Leffler function 
is an  analytic function on $\C$ and given by the power series
\begin{align}
  \label{eq:mittag_leffler}
  e_{\alpha,\mu}(z):=\sum_{n=0}^{\infty}{\frac{z^n}{\Gamma(n\alpha+\mu)}}.
\end{align}
We collect some important properties we will need later on.
We start with the following decomposition result, also giving us asymptotic estimates.
\begin{proposition}
  \label{prop:decompose_ml}
  For $0 < \alpha < 2, \mu \in \R$ and $\frac{\alpha \pi}{2} < \zeta < \alpha \pi$,
  we can decompose the Mittag-Leffler function as
\begin{align}
  \label{eq:mittag_leffler_decomp}
  e_{\alpha,\mu}(z)
  &= - \sum_{n=1}^{N}{\frac{1}{\Gamma(\mu-\alpha n)}\frac{1}{z^n}}
    + R_{\alpha,\mu}^{N}(z)
   \qquad \text{for } \zeta \leq \abs{\operatorname{Arg}{z}} \leq \pi.
\end{align}
where $R_{\alpha,\mu}^{N}$ is analytic away from zero and satisfies
\begin{align}
  \abs{R_{\alpha,\mu}^{N}(z)}\leq C\, \Gamma(\alpha N) \abs{z}^{-(N+1)} \qquad \forall \abs{z} \geq z_0 > 0
\end{align}
for a constant $C>0$ depending only on $z_0$ and $\zeta$.
\end{proposition}
\begin{proof}
  The statement can be found in~\cite[Eqn 1.8.28]{ksh06} where the dependence of the remainder term
  with on $N$ is not made explicit. To get the explicit estimate on the remainder,
  we follow~\cite[Section 18.1]{emt81}. There, it is proven that the remainder can be written as
  \begin{align*}
    R_{\alpha,\mu}^{N}(z)
    &=\frac{ z^{-N-1}}{2\pi \ii} \int_{\widetilde{\CC}}{\Big(1-\frac{t^{\alpha}}{z}\Big)^{-1} t^{(N+1)\alpha -\mu} e^{t} \; dt},
  \end{align*}
  where $\widetilde{\CC}$ can be taken as two rays $\{r \zeta_0: \, r\geq 1\}$, $\{r \overline{\zeta_0}: \, r\geq 1\}$ and a small circular arc connecting the two
  without crossing the negative real axis. $\zeta_0$ is taken in the left half-plane such that the opening angle of $\widetilde{\CC}$ is sufficiently
  large in order to avoid possible poles of the integrand and ensure that the  term $(1-t^{\alpha}/z)^{-1}$ is uniformly bounded.
  The stated result then follows easily by comparing the integral under consideration to the definition of the Gamma function.
\end{proof}
Setting $N=1$ in Proposition~\ref{prop:decompose_ml} and simple calculation yields the following estimates:
\begin{align}
  \label{eq:mittag_leffler_est}
  \abs{e_{\alpha,\mu}(z)}\leq \frac{C}{1+\abs{z}^{-s}}
  \quad \text{for } \zeta \leq \abs{\operatorname{Arg}{z}} \leq \pi,\, s \in [0,1]
\end{align}

For $\alpha=\mu=1$, the Mittag-Leffler function $e_{1,1}$ is the usual exponential function.
For the decomposition result, we can skip
the terms involving powers $z^{-n}$ in this case as $e^{z}$ already decays faster than any polynomial.

Finally, we need a way of computing antiderivatives of the convolution kernel in
\eqref{eq:parabolic_representation_formula}.
\begin{proposition}
  For $n \in \N_0$, $\alpha>0$, $z \in \C \setminus \{0\}$,  $\lambda \in \C$, it holds that
  \begin{align}
    \label{eq:diff_of_ml}
    z^{\alpha-1}e_{\alpha,\alpha}(\lambda \, z^{\alpha} )&=
    \Big(\frac{\partial}{\partial z}\Big)^{n} \Big(z^{\alpha+n-1} e_{\alpha,\alpha+n}(\lambda \,z^{\alpha})  \Big).
  \end{align}
\end{proposition}
\begin{proof}
  Follows from~\cite[Eqn.~1.10.7]{ksh06} by taking $\beta:=\alpha+n$.
\end{proof}

\subsection{Double exponential quadrature for the parabolic problem}

\paragraph{The case of finite regularity}
In this section, we investigate the convergence of our method in the case that $u_0$ and
$f$ have finite $\mathbb{H}^{2\rho}$-regularity for some $\rho \geq 0$. It will showcase
most of the new ingredients needed to go from the elliptic case to the time dependent one while keeping
the technicalities to a minimum.  The step towards Gevrey-regularity will then mainly consist of carefully retracing the
argument and fine-tuning parameters.

\begin{lemma}
  \label{lemma:parabolic_full_discretization_hom_finte_regularity}
  Assume that either $\alpha+\beta<2$ or $\sigma=1$ (i.e., the
  case $\alpha=\beta=1$ and $\sigma=1/2$ is not allowed).  
  Let $u(t):=e_{\alpha,\mu}(-t^{\alpha} \LL^{\beta}) u_0$ 
  and assume $u_0 \in \mathbb{H}^{2\rho}(\Omega)$ for some $\rho >0$.
  Let $u_k:=Q^{\LL}\big(e_{\alpha,\mu}(-t^{\alpha} z^{\beta}),\Nquad\big) u_0$
  be the corresponding discretization using stepsize $k>0$
  and $\Nquad \in \N$ quadrature points.
    
  Then, the following estimate holds for all $\eta\geq 1$ and $r\in [0,\beta/2]$:
  \begin{multline*}
      \norm{u(t)-u_{k}(t)}_{\mathbb{H}^{2r}(\Omega)} 
      \lesssim t^{-\eta \alpha} \Big(
      e^{
        -\min{\big\{p(\sigma,\theta)\sqrt{\beta+\rho-r}
          ,\sqrt{\eta} \,\gamma_1\big\}}
      \frac{1}{\sqrt{k}}}
    +  e^{-\frac{\gamma}{k}}\Big) \norm{u_0}_{\mathbb{H}^{2\rho}} \\
    + t^{-\alpha/2} \exp(-\gamma e^{k \Nquad})  \norm{u_0}_{\mathbb{H}^{\max(2\rho,\frac{1}{4})}(\Omega)}.
    \end{multline*}
    Here $\gamma_1$ is the constant from Corollary~\ref{cor:sinc_in_dexp_for_quadrature_operators}.
    The implied constant and $\gamma$ may depend on $r$, the smallest eigenvalue $\lambda_0$ of $\LL$, $\beta$,$\alpha$, $\kappa$, $\theta$, $\sigma$ and $\rho$.
\end{lemma}
\begin{proof}
  We start with $\Nquad=\infty$ and
  split the Mittag-Leffler function according to \eqref{eq:mittag_leffler_decomp}. We write
  \begin{align}
    \label{eq:decomp_of_ml_quadrature_error}
    E^{\LL}\Big(e_{\alpha,\mu}(-t^{\alpha} z^{\beta}) \Big)
    &= \sum_{n=1}^{N}{\frac{ (-1)^n \,t^{-\alpha n} }{\Gamma(\mu-\alpha n)} E^{\LL}(z^{-\beta n})}
    + E^{\LL}\big(R^N_{\alpha,\mu}(-t^{\alpha} z^{\beta})\big).
  \end{align}
  For the first terms, we apply Lemma~\ref{cor:elliptic_full_discretization},
  and for the final term we  use the decay estimate~\eqref{eq:mittag_leffler_decomp}
  and Corollary~\ref{cor:sinc_in_dexp_for_quadrature_operators}.
  Note that this is where we have to exclude the case $\alpha=\beta=1$ and $\sigma=1/2$.
  If $\alpha<1$ the Mittag-Leffler function is contractive on a large enough sector.
  If $\beta<1$, the map $z\mapsto z^{\beta}$ maps the required sector into the right half plane.
  Otherwise, the exponential function only decays in the right half-plane, not any slightly bigger sector.
  Thus, if $\sigma=1/2$, Corollary~\ref{cor:sinc_in_dexp_for_quadrature_operators} does not apply.

  Overall, we get the estimate:
  \begin{align*}
    \norm{E^{\LL}\Big(e_{\alpha,\mu}(t^{\alpha} z^{\beta}) \Big)}_{\mathbb{H}^{2 r}(\Omega)}
    &\lesssim
      \sum_{n=1}^{N-1}{\frac{ t^{-\alpha n} }{\Gamma(\mu-\alpha n)}
      \exp\Big(-
      \min\Big\{\frac{p(\sigma,\theta) \sqrt{\beta n + \rho -r}}{\sqrt{k}},
      \frac{\gamma}{k}\Big\}
      \Big)\norm{u_0}_{\mathbb{H}^{2\rho}(\Omega)}} \\
    &+ \Gamma(\alpha N) t^{-\alpha N} \exp\Big(-\gamma_1 \frac{\sqrt{N}}{\sqrt{k}}\Big)
      \norm{u_0}_{\mathbb{H}^{\min(\rho,1)}(\Omega)}.
  \end{align*}
  If $\eta$ is an integer, we can pick $N=\eta$ to get the statement for 
  $\Nquad=\infty$.   For general $\eta\geq 1$, we can interpolate between $\lfloor\eta\rfloor$
  and $\lceil \eta \rceil$.
  The treatment of the cutoff error follows  as in
  Corollary~\ref{cor:sinc_in_dexp_for_quadrature_operators},
  exploiting that $e_{\alpha,\mu}(z)$ decays like~\eqref{eq:mittag_leffler_est} with $s:=\beta/2$.
\end{proof}
Picking $\eta$ large enough,
Lemma~\ref{lemma:parabolic_full_discretization_hom_finte_regularity} shows that
for fixed times $t>0$ we get the same convergence rate as for the elliptic problem, though the error
deteriorates as $t$ gets small.

Now that we understand the homogeneous problem,
we can look at the case of allowing inhomogeneous right-hand sides $f$ by using the
representation formula~\eqref{eq:parabolic_representation_formula}. We point out that naive application
of Corollary~\eqref{eq:cor:de_for_elliptic_int1} also inside the time-convolution integral
would fail to give good rates, as the error may blow up faster than $\tau^{-\alpha}$ for small times,
leading to a non-integrable function. Instead, we have the following Theorem:
\begin{theorem}
  \label{lemma:parabolic_inhomogeneous_finite_regularity}
  Assume that either $\alpha+\beta<2$ or $\sigma=1$ (i.e., the
    case $\alpha=\beta=1$ and $\sigma=1/2$ is not allowed).  
    Let $u$ be the solution to~\eqref{eq:parabolic_model_problem}.   
    Assume $u_0 \in \mathbb{H}^{2\rho}(\Omega)$ for some $\rho > 0$,
    and $f \in C^{m}([0,T],  \mathbb{H}^{2\rho}(\Omega))$ for some
    $m\in \N$.
    Let $u_k$ be the corresponding discretization using stepsize $k>0$
    and $\Nquad \in \N$ quadrature points as defined in~\eqref{eq:2}.
  
    Then, the following estimate holds for all
     $t \in (0,T)$, $r\in [0,\beta/2]$  and any $q<1$:
    \begin{multline*}
      \norm{u(t)-u_{k}(t)}_{\mathbb{H}^{2r}(\Omega)}   
      \lesssim 
      \Big( t^{-m+\alpha(1-q)} e^{-\frac{\min\big\{p(\sigma,\theta) \sqrt{\beta+\rho-r},\gamma_1\sqrt{m/\alpha +q}\big\}}{\sqrt{k}}}
      + t^{-m}e^{-\gamma/k}
      + t^{-\alpha/2} e^{-\gamma e^{k \Nquad}}\Big) \\
      \times \Big(\norm{u_0}_{\mathbb{H}^{2\rho}(\Omega)} +
        \sum_{j=0}^{m}{\max_{\tau \leq t}\|{f^{(j)}(\tau)}\|_{\mathbb{H}^{2\rho}(\Omega)}}\Big)  
    \end{multline*}
    $\gamma_1$ is the constant from Corollary~\ref{cor:sinc_in_dexp_for_quadrature_operators}.
    The impied constant and $\gamma$  may depend on $q$, $r$, $T$, the smallest eigenvalue $\lambda_0$ of $\LL$, $\beta$, $\alpha$, $\kappa$, $\theta$, $\sigma$ and $\rho$.    
\end{theorem}
\begin{proof}
  As we have already estimated the error of the homogeneous part,
  we only consider the part corresponding to the inhomogenity, i.e., for now
  let $u_0=0$.
  We integrate by parts $m$ times, using~\eqref{eq:diff_of_ml}:
  \begin{multline*}
    \int_{0}^{t} { \tau^{\alpha-1} e_{\alpha,\alpha}\big(-\tau^{\alpha} \lambda^\beta\big) f(t-\tau)\,d\tau}
    =
    {\sum_{j=1}^{m}
      { t^{\alpha+j-1} e_{\alpha,\alpha+j}(- t^{\alpha} \lambda^{\beta} ) f^{(j-1)}(0)}}\\
    +{\int_{0}^{t} { \tau^{\alpha+m-1} e_{\alpha,\alpha+m}\big(-\tau^{\alpha} \lambda^\beta\big) f^{(m)}(t-\tau)\,d\tau}} 
  \end{multline*}

  Transferring this identity to the operator-valued setting,
  this means that we can analyze the quadrature error for these terms separately.
  \begin{align}
    \norm{u(t)-u_k(t)}_{\mathbb{H}^{2r}(\Omega)}
    &=\Big\|\int_{0}^{t} {
      \tau^{\alpha-1}E^\LL\Big(e_{\alpha,\alpha}\big(-\tau^{\alpha} z^\beta\big)\Big) f(t-\tau)\,d\tau}
      \Big\|_{\mathbb{H}^{2r}(\Omega)}
    \nonumber\\ 
    &\leq
      \begin{multlined}[t][10.5cm]
      {\sum_{j=1}^{m}
      { t^{\alpha+j-1} \norm{ E^{\LL}\Big(e_{\alpha,\alpha+j}(- t^{\alpha}z^{\beta} ) \Big)f^{(j-1)}(0)}_{\mathbb{H}^{2r}(\Omega)}}}
    \\+       
      {\int_{0}^{t} { \tau^{\alpha+m-1}
          \norm{E^{\LL}\big(e_{\alpha,\alpha+m}\big(-\tau^{\alpha} z^\beta\big)\big) f^{(m)}(t-\tau)}_{\mathbb{H}^{2r}(\Omega)}\,d\tau}}.
    \end{multlined}
    \label{eq:parabolic_inh_decomposition}    
  \end{align}
All the terms appearing are of the structure in Lemma~\ref{lemma:parabolic_full_discretization_hom_finte_regularity}.
  Most notably, the first $m$ terms are evaluated at a fixed $t>0$ thus we don't have to analyze them further
  and can just accept some $t$-dependence.

  Investigating the remaining integral, we get
  by using $\eta:=m/\alpha+q$ in  Lemma~\ref{lemma:parabolic_full_discretization_hom_finte_regularity}:
  \begin{multline*}
\int_{0}^{t} { \tau^{\alpha+m-1}
    \norm{E^{\LL}\big(e_{\alpha,\alpha+n}(-\tau^{\alpha} z^\beta)\big) f^{(m)}(t-\tau)}_{\mathbb{H}^{2r}(\Omega)}\,d\tau} \\
  \lesssim
      \int_{0}^{t}{\tau^{\alpha+m-1-m-\alpha q} 
      \Big[\!\exp\Big(-\frac{ \min\{p(\sigma,\theta) \sqrt{\beta+\rho-r},\gamma_1\sqrt{\eta}\}}{\sqrt{k}}\Big)
      \!+\!e^{-\frac{\gamma}{k}} \Big]\big\|{f^{(m)}(t-\tau)}\big\|_{\mathbb{H}^{2\rho}(\Omega)}  d\tau}.
    \end{multline*}
    For $q<1$, this is an integrable function
    (with respect to $\tau$). The dominating $t$-dependence can be found in the
    first term of~\eqref{eq:parabolic_inh_decomposition} and thus the result follows.
    The cutoff error is treated  like before, making use of the decay of $e_{\alpha,\alpha}$.
\end{proof}

\begin{remark}
  Corollary~\ref{lemma:parabolic_inhomogeneous_finite_regularity} shows that, as long as
  we assume that $f$ is smooth enough in time we recover the same convergence rate
  $p(\sigma,\theta)\sqrt{\beta+\rho-r}$ as in the homogeneous
  and elliptic case.
\end{remark}

\paragraph{The case of Gevrey-type regularity}
If the data not only satisfies some finite regularity estimates but instead is even
in some Gevrey-type class of functions, we can again improve the convergence
rate, and almost get rid of the square root in the exponent.
We go back to the homogeneous problem and assume
that $k<1/2$ so that the logarithmic terms can be written down succinctly.
\begin{lemma}
  \label{lemma:parabolic_full_discretization:super_smooth}
  Assume that either $\alpha+\beta<2$ or $\sigma=1$ (i.e., the
  case $\alpha=\beta=1$ and $\sigma=1/2$ is not allowed).  
  Let $u(t):=e_{\alpha,\mu}(-t^{\alpha} \LL^{\beta}) u_0$ 
  and assume
  that  there exist constants $C_{u_0} , \omega, R_{u_0}>0$ such that
  \begin{align*}
    \norm{u_0}_{\mathbb{H}^{\rho}(\Omega)} &\leq C_{u_0} \, R_{u_0}^{\rho} \,\big(\Gamma(\rho+1)\big)^{\omega}  < \infty \qquad  \forall \rho\geq 0.
  \end{align*}
  Let $u_k(t):=Q^{\LL}\big(e_{\alpha,\mu}(-t^{\alpha} z^{\beta}),\Nquad\big) u_0$
  be the  discretization of $u$ using stepsize $k\in(0,1/2)$
  and $\Nquad \in \N$ quadrature points.
  Then, the following estimate holds:
  \begin{align*}
      \norm{u(t)-u_{k}(t)}_{\mathbb{H}^{\beta}(\Omega)}
    &\lesssim C_{u_0}
      \exp\big(-\frac{\gamma}{k\abs{\ln(t_\star)}\abs{\ln{k}}}\big) + C_{u_0}t^{-\alpha/2}\exp\big(- \gamma e^{k \Nquad}\big).
  \end{align*}
  with $t_{\star}:=\min(t,1/2)$.
  The implied constant and $\gamma$ may depend on $\varepsilon$,
  the smallest eigenvalue $\lambda_0$ of $\LL$, $\beta$, $\alpha$, $\kappa$, $\theta$, $\sigma$,  $R_{u_0}$  and $\omega$.
\end{lemma}
\begin{proof} 
  We go back to~\eqref{eq:decomp_of_ml_quadrature_error}, but apply Corollary~\ref{cor:elliptic_full_discretization:super_smooth} to each of the first $N$ terms, getting:
  \begin{equation}
        \label{eq:decomp_of_ml_hom_smooth}
        \begin{multlined}[c][0.9\textwidth]
          \norm{E^{\LL}\Big(e_{\alpha,\mu}(-t^{\alpha} z^{\beta}) \Big) u_0}_{\mathbb{H}^{\beta}(\Omega)}
          \lesssim
          C_{u_0}\sum_{n=1}^{N}{\frac{ t^{-\alpha n} }{\Gamma(\mu-\alpha n)}
            \exp\Big(-\frac{\gamma}{k \abs{\ln(k)}}\Big)
      } \\
      + \Gamma(\alpha N) t^{-\alpha N} \exp\Big(-\gamma \frac{\sqrt{N-2\varepsilon}}{\sqrt{k}}\Big)
    \norm{u_0}_{L^2(\Omega)}.
  \end{multlined}
\end{equation}
We  estimate the first terms by
  \begin{align*}
    \frac{ t^{-\alpha n} }{\Gamma(\mu-\alpha n)}
    \exp\Big(-\frac{\gamma}{k \abs{\ln(k)}}\Big)
    &\lesssim
      \exp\Bigg(-\alpha \ln(t) n + c_1 n \log(n)  -\frac{\gamma}{k \abs{\ln(k)}}\Bigg).
  \end{align*}
  For $n \leq \frac{\delta}{k\abs{\ln(t_{\star})}^2\abs{\ln(k)}^2}$, we can estimate the exponent by
  \begin{align*}
    -\frac{1}{k \abs{\ln(t_\star)}\abs{\ln(k)}}
    \Bigg[
      {-\frac{\alpha \delta }{\abs{\ln(k)}} }
         - \frac{  c_1 \delta }{\abs{\ln(t_\star)}\abs{\ln(k)}} \ln\Big(\frac{\delta}{\ln(t_\star)^2 \ln(k)^2 k}\Big)
      +\gamma \ln(2) \Bigg]
  \end{align*}
  For $\delta $ small enough, depending on $c_1$, $\alpha$ and $\gamma$,
  the term in brackets is uniformly positive (i.e., independently of $t$ and $k$), we can thus estimate
  for some $\gamma_1>0$:
  \begin{align*}
      \frac{ t^{-\alpha n} }{\Gamma(\mu-\alpha n)}
    \exp\Big(-\frac{\gamma}{k \abs{\ln(k)}}\Big)
    &\lesssim e^{-\frac{\gamma_1}{k\abs{\ln(t_\star)} \abs{\ln(k)}}}.
  \end{align*}
  The remainder term  behaves like
  \begin{align*}
    t^{-\alpha N} \Gamma(\alpha N) \exp\Big(-\gamma \frac{\sqrt{N-2\varepsilon}}{\sqrt{k}}\Big)
    &\lesssim
      \exp\Big(-\alpha N \ln(t) - c_2 N \ln(N) - \gamma \frac{\sqrt{N-2\varepsilon}}{\sqrt{k}}\Big).
  \end{align*}
  By picking $N=\big\lceil{\frac{\delta}{k\abs{\ln(t_\star)}^2\abs{\ln(k)}^2}}\big\rceil$, 
  the exponent be bounded up to a constant by 
  \begin{align*}
      -\frac{\sqrt{\delta}}{k \abs{\ln(t_\star)} \abs{\ln(k)}}
      \Bigg[
      {-\frac{\gamma \sqrt{\delta}}{\abs{\ln(k)}}
    -  \frac{c_1\sqrt{\delta}}{\abs{\ln(t_\star)\ln(k)}} \ln\Big(\frac{\delta}{k \abs{\ln(t_\star)}^2 \abs{\ln(k)}^2}\Big)
    + \gamma \ln(2)
      }
      \Bigg].
  \end{align*}
  By taking the factor $\delta$ sufficiently small, we get that the term in brackets stays uniformly positive,
  which shows
  \begin{align*}
    \norm{E^{\LL}\Big(e_{\alpha,\mu}(-t^{\alpha} z^{\beta}) \Big)}_{\mathbb{H}^{\beta}(\Omega)}
    &\lesssim
      \exp\Big({-\frac{\gamma}{\abs{\ln(t_\star)}\abs{\ln(k)}k}}\Big).
  \end{align*}
  The cutoff error can easily be dealt with as in the previous results, as the Mittag-Leffler function
  satisfies the decay bound \eqref{eq:mittag_leffler_est} for $s=1/2$.
\end{proof}

Finally, we are in a position to also include the inhomogenity $f$ into our treatment.
Just as in Lemma~\ref{lemma:parabolic_inhomogeneous_finite_regularity}, we
use integration by parts to decompose the error into parts for positive
times and a remainder integral with ``nice enough'' behavior with respect to $\tau$.

\begin{theorem}
  \label{thm:convergence_parabolic_super_smooth_full}
  Assume that either $\alpha+\beta<2$ or $\sigma=1$ (i.e., the
  case $\alpha=\beta=1$ and $\sigma=1/2$ is not allowed).  
  Let $u$ be the solution to~\eqref{eq:parabolic_model_problem}, and assume that the
  data satisfy
  \begin{align}
    \label{eq:gevrey_parabolic}
    \begin{split}
    \norm{u_0}_{\mathbb{H}^{\rho}(\Omega)}
    &\leq C_{u_0} \, R_{u_0}^{\rho} \,\big(\Gamma(\rho+1)\big)^{\omega}  < \infty  \;\;\qquad \qquad  \forall \rho\geq 0 \\
    \big\|{f^{(n)}(t)}\big\|_{\mathbb{H}^{\rho}(\Omega)}
    &\leq C_{f} \; R_{f}^{\rho+n} \,\big(\Gamma(\rho+1)\big)^{\omega} \big(n!\big)^{\omega}  < \infty \qquad
    \forall t \in [0,T],\; \forall  \rho\geq 0,\;\forall n \in \N_0 .
  \end{split}      
  \end{align}

  Let $u_k$ be the corresponding discretization using stepsize $k \in (0,1/2)$
  and $\Nquad \in \N$ quadrature points as defined in~\eqref{eq:2}.
  Then, the following estimate holds:
  \begin{align*}
      \norm{u(t)-u_{k}(t)}_{\mathbb{H}^{\beta}(\Omega)}
    &\lesssim (1+t)
      \exp\Big(-\frac{\gamma}{k\abs{\ln{t_\star}}\abs{\ln{k}}}\Big) + \big(t+ t^{-\gamma/2}\big)\exp\big(- \gamma e^{k \Nquad}\big)
  \end{align*}
  with $t_\star:=\min(t,1/2)$.
  The implied constant and $\gamma$ may depend on  the smallest eigenvalue $\lambda_0$ of $\LL$, $\beta$, $\theta$, $\sigma$ and the constants from~\eqref{eq:gevrey_parabolic}.
\end{theorem}
\begin{proof}
  We again work under the assumption $u_0=0$ and focus on the error when dealing with
  the inhomogenity $f$ alone and also start with $\Nquad=\infty$. We also for now
  take $t\leq 1$.
  
  Going back to~\eqref{eq:parabolic_inh_decomposition}  we get
  for $N \in \N_0$ to be fixed later
  \begin{equation}
    \label{eq:parabolic_inh_disc_smooth_step1}
    \begin{multlined}[c][14cm]
      \norm{u(t)-u_k(t)}_{\mathbb{H}^{2r}(\Omega)}
      \leq
      {\sum_{j=1}^{N}
      { t^{\alpha+j-1} \norm{ E^{\LL}\Big(e_{\alpha,\alpha+N}(-t^{\alpha} z^{\beta}) \Big)f^{(j-1)}(0)}_{\mathbb{H}^{\beta}(\Omega)}}}
    \\+       
      {\int_{0}^{t} { \tau^{\alpha+N-1}
          \norm{E^{\LL}\big(e_{\alpha,\alpha+n}\big(-\tau^{\alpha} z^\beta\big)\big) f^{(N)}(t-\tau)}_{\mathbb{H}^{\beta}(\Omega)}\,d\tau}}.
    \end{multlined}
  \end{equation}
  For the first terms, we apply Lemma~\ref{lemma:parabolic_full_discretization:super_smooth} to
  get exponential convergence, as long as $f^{(j)}$ is in the right Gevrey-type class.
  Namely, we note that we can estimate
  \begin{align*}
    \norm{{f}^{(n)}(t)}_{\mathbb{H}^{2\rho}(\Omega)}
    \lesssim e^{\widetilde{\omega} N \ln(N)} e^{\widetilde{\omega} \rho \ln(\rho)}
  \end{align*}
  by possibly  tweaking $\widetilde{\omega}$ compared to $\omega$.  This allows us to estimate
  \begin{align}
    \label{eq:parabolic_inh_disc_smooth_step2}
    {\sum_{j=1}^{N}
    { t^{\alpha+j-1} \norm{ E^{\LL}\Big(e_{\alpha,\alpha+j}(-  t^{\alpha} z^{\beta} ) \Big)f^{(j-1)}(0)}_{\mathbb{H}^{2r}(\Omega)}}}
    &\lesssim e^{\widetilde{\omega}N \ln(N)} e^{-\frac{\gamma}{\abs{\ln(t_\star)} \abs{\ln(k)} k}}.
  \end{align}
  Again restricting $\widetilde{\omega}$ to absorb the factor $N$ due to the summation.

  For the remainder in~\eqref{eq:parabolic_inh_disc_smooth_step1}, we
  look at the pointwise error at fixed $0<\tau<t$,
  shortening $\widetilde{f}^{(N)}:=f^{(N)}(t-\tau)$.  
  Going back to~\eqref{eq:decomp_of_ml_hom_smooth}, we can use the
  additional powers of $t$ to get rid of the $\ln(t)$ term in the exponential:
  \begin{align*}
    \tau^{\alpha+N-1}\norm{E^{\LL}\big(e_{\alpha,\alpha+n}\big(-\tau^{\alpha} z^\beta\big)\big)
    \widetilde{f}^{(N)}}_{\mathbb{H}^{\beta}(\Omega)}
    &\lesssim
      \sum_{n=1}^{N-1}{C_{f} e^{\widetilde{\omega} N \ln(N)}\frac{ \tau^{-\alpha(n-1) + N -1} }{\Gamma(\mu-\alpha n)}
      \exp\Big(-\frac{\gamma}{k \abs{\ln(k)}}\Big)
      } \\
    &+ \Gamma(\alpha N) \tau^{(1-\alpha)(N-1)}
      C_f e^{\widetilde{\omega} N \ln(N)} \exp\Big(-\gamma \frac{\sqrt{N-2\varepsilon}}{\sqrt{k}}\Big).
  \end{align*}
  We then proceed as in the proof of Lemma~\ref{lemma:parabolic_full_discretization:super_smooth},
  noting that since the $\tau$-dependent terms can be bounded independently of $N$ we can
  get by without the $\ln(t_\star)$-term in the exponent. Overall, we get by
  tuning $N \sim \delta/(\abs{\ln(k)}^2 k)$ (also in \eqref{eq:parabolic_inh_disc_smooth_step2})
  appropriately:
  \begin{align*}
      \norm{u(t)-u_k(t)}_{\mathbb{H}^{2r}(\Omega)}
    &\lesssim 
       \exp\Big(-\frac{\gamma}{k \abs{\ln(t_\star)} \abs{\ln(k)}}\Big)
      + \int_{0}^{t}{\exp\Big(-\frac{\gamma}{k \abs{\ln(k)}}\Big) \,d\tau}.
  \end{align*}
  which easily gives the stated result.
  If $t>1$, we can skip the integration by parts step for the integration over $(1,t)$ and directly
  apply Lemma~\ref{lemma:parabolic_full_discretization:super_smooth}.
  The cutoff error is treated as always.  
\end{proof}

\section{Numerical examples}
\label{sect:numerics}
In this section, we investigate, whether the theoretical results obtained in Sections~\ref{sect:elliptic} and
\ref{sect:parabolic} can also be observed in practice.
We compare the following quadrature schemes:
\begin{enumerate}[(i)]
\item DE1: double exponential quadrature using $\sigma=1/2$ and $\theta=4$,
\item DE2: double exponential quadrature using $\sigma=1$ and $\theta=4$,
\item DE3: double exponential quadrature using $\sigma=1$ and $\theta=1$,
\item sinc: standard sinc quadrature
\item Balakrishnan: a quadrature scheme based on the Balakrishnan formula
\end{enumerate}

For the double exponential quadrature schemes,
we used $k=0.9\ln(\Nquad)/\Nquad$. This makes the
cutoff error decay like $e^{- \Nquad^{0.9}}$, which is sufficiently fast to not impact
the overall convergence rate. The factor $0.9$ was observed to have some slightly improved stability
compared to $1$.

For the standard sinc-quadrature, the proper tuning of $k$ and $\Nquad$ is more involved.
Following~\cite{bonito_pasciak_parabolic}, we picked
$k=\sqrt{\frac{2\pi d}{\beta\Nquad}}$ with $d=\pi/5$. 
The Balakrishnan formula is only possible for the elliptic problem. It is described in detail
in~\cite{blp_accretive}. Following~\cite[Remark 3.1]{blp_accretive} we used
\begin{align*}
  k:=\sqrt{\frac{\pi^2}{1.8\beta N}} \qquad
  M:=\left\lceil \frac{\pi^2}{2(1-\beta)k^2}\right\rceil,
\end{align*}
where $M$ is the number of negative quadrature points. This corresponds
(in their notation) to taking $s^+:=\beta/10$, which
was taken because it yielded good results.

\subsection{The pure quadrature problem}
In this section, we focus on a scalar quadrature problem. Namely, we investigate
how well our quadrature scheme can approximately evaluate the following functions
using the Riesz-Dunford calculus
(a) $z^{-\beta}$ and (b) $e_{\alpha,1}(-t^{\alpha} z^{\beta})$
at different values $\lambda \in (4,\infty)$.
This is equivalent to solving the elliptic and parabolic problem with data consisting of
a single eigenfunction corresponding to the eigenvalue $\lambda$.
Throughout, we used $\kappa:=3$.
Theoretical investigations revealed, that the quadrature error is largest at $\ln(\lambda) \sim k^{-1/2}$
(see the proof of Corollary~\ref{cor:de_for_elliptic:lambda_robust}).
Therefore, we make sure
that for each $k$ under consideration, such a value of $e^{\frac{1}{\sqrt{k}}}$ is among the $\lambda$-values sampled.
More precisely, the sample points consist of
\begin{align*}
  \bigcup_{\Nquad=1}^{N_{\max}}\big\{   5+\exp\big(2\sqrt{\beta/k(\Nquad)}\big)  \big\} \cup  \big\{5+\exp(\beta/k(\Nquad)) \big\} \qquad
  \text{with } \qquad k(\Nquad)=0.9 \ln(\Nquad)/\Nquad,
\end{align*}
and we consider the maximum error over all these samples. We used $t:=1$ for all experiments.

\begin{figure}[htb]
  \centering
  \begin{subfigure}{0.46\textwidth}
    \includeTikzOrEps{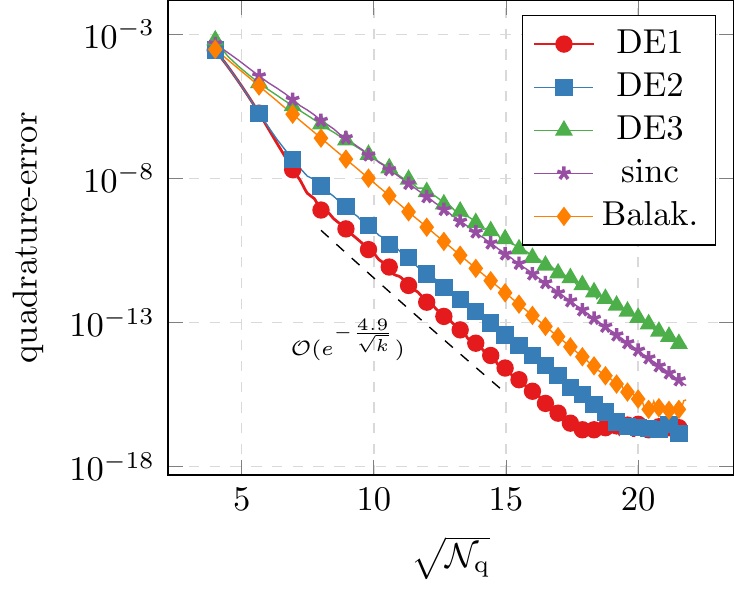}
    \vspace{-1\baselineskip}
  \caption{$f(z)=z^{-0.6}$.}    
  \label{fig:comp_pow}
\end{subfigure}
\;
  \begin{subfigure}{0.46\textwidth}
    \includeTikzOrEps{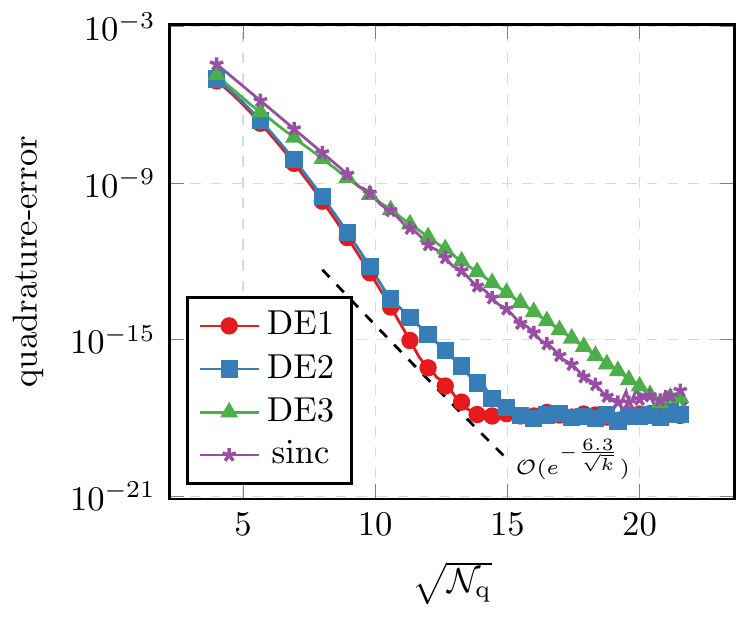}
    \vspace{-1\baselineskip}
  \caption{$f(z)=z^{-1}$.}    
  \label{fig:comp_pow2}
\end{subfigure}
\\
  \begin{subfigure}{0.46\textwidth}
    \includeTikzOrEps{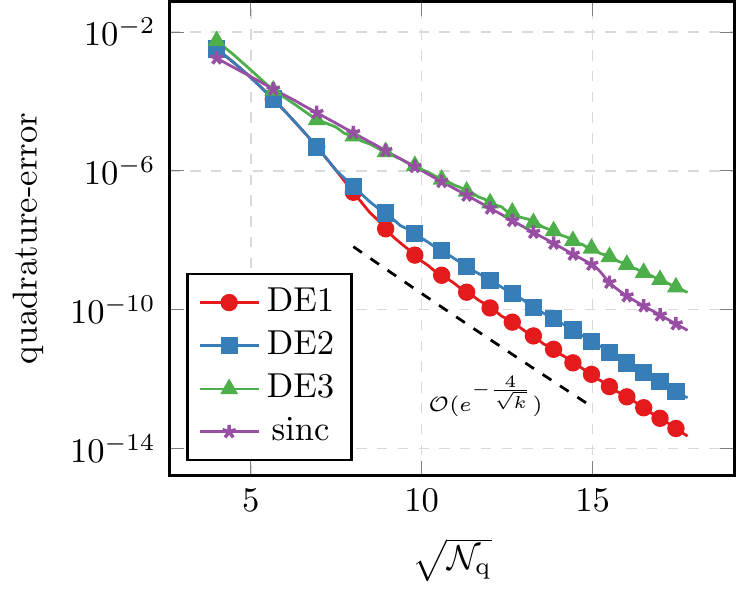}
    \vspace{-1\baselineskip}
  \caption{$e_{\alpha,1}(- z^{\beta})$,
    with $\alpha=0.25$, $\beta=0.4$.}    
  \label{fig:comp_ml1}
\end{subfigure}
\;
  \begin{subfigure}{0.46\textwidth}
    \includeTikzOrEps{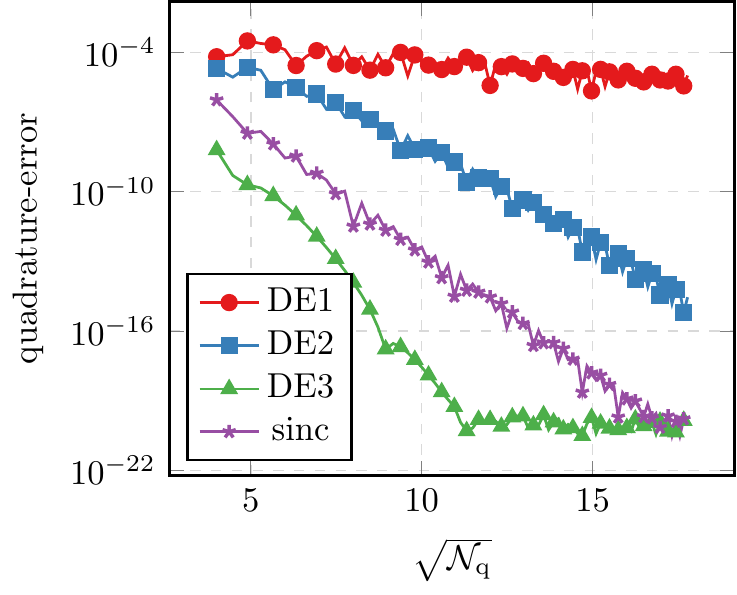}
        \vspace{-1\baselineskip}
        \caption{ $e_{\alpha,1}(- z^{\beta})$,
          with $\alpha=1$, $\beta=1$.}    
  \label{fig:comp_ml2}
\end{subfigure}
\caption{Comparison of quadrature schemes -- scalar problem}
\end{figure}

We observe that for the most part, choosing $\sigma=1/2$ and $\theta$ moderately large gives the
best result. This agrees with our theoretical findings. This method fails to converge though
if $\alpha=\beta=1$ is chosen as the parameters for the Mittag-Leffler function. This also
agrees with the theory, because in this case, $\psits$ fails to map into the domain where
$e_{\alpha,\mu}$  is decaying (see~\eqref{eq:mittag_leffler_est}). This shows
that the restriction on $\sigma$ in the theorems of Section~\ref{sect:parabolic} is
necessary. If we only consider the elliptic problem, no such restriction is necessary, as the
decay property is valid on all of the complex plane. All the other methods perform well in all of
the cases. The straight-forward double exponential formula, i.e. $\sigma=\theta=1$,
is often outperformed by the simple sinc quadrature scheme, (except in the $\alpha=\beta=1$ case of the
exponential).
For comparison, we've included the (rounded) predicted rate for the DE1 scheme in the plots.
We observe that for several applications our estimates appear sharp. For $f(z)=z^{-1}$
the scheme outperforms the prediction, but this might be due to a large preasymptotic regime.
We note that for $e^{-z^{\beta}}$, we expect better estimates than the ones presented
in this article to be possible due to the exponential decay. This is also true
for the standard sinc methods, see~\cite{blp17}.

Second, we look at the case of a single frequency $\lambda$ and see how the convergence rate decays as $\lambda \to \infty$.
In order to better see the $\lambda$-dependence of the quadrature error,
we consider the relative error of the quadrature, i.e., we look at
$
E^{\lambda}(z^{-\beta})/\lambda^{-\beta}
$
for $\beta=0.5$.
The theory from Theorem~\ref{thm:de_for_elliptic} predicts behavior of the form
$e^{-\frac{\gamma}{\ln(\lambda)k}}$, i.e., the rate drops like $\ln(\lambda)$.
In Figure~\ref{fig:lambda_dependence}, we can see this behavior quite well. In
comparison, using standard $\sinc$ quadrature gives a $\lambda$-robust asymptotic
rate, but only of order $\sqrt{\Nquad}$.
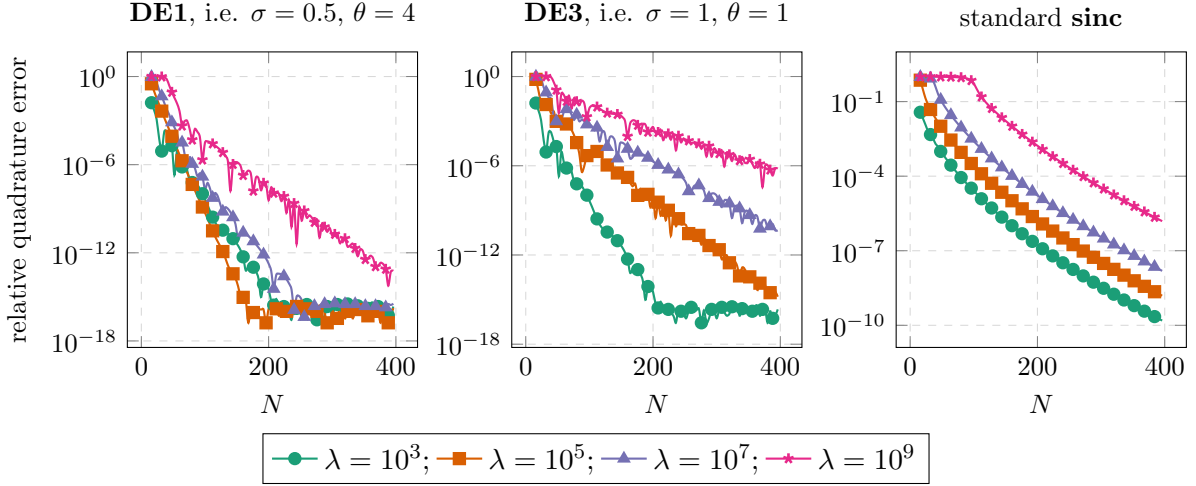
\begin{figure}
  \centering  
  \setlength\tabcolsep{0 pt}
  \begin{tabular}{@{}ccc@{}}
\begin{tikzpicture}
  \pgfplotsset{
        cycle list/Dark2,
        cycle multiindex* list={
            mark list\nextlist
            Dark2\nextlist
        },
    }
  \begin{semilogyaxis}[
    width=5.4cm, height=5.5cm,     
    grid = major,
    title={{\small{\textbf{DE1}, i.e. $\sigma=0.5$, $\theta=4$}}},
    grid style={dashed, gray!30},
    axis background/.style={fill=white},
    xlabel=$N$,
    ylabel=relative quadrature error,
    legend columns=-1,
    legend entries={$\lambda=10^3$;,$\lambda=10^5$;,$\lambda=10^7$;,$\lambda=10^9$},
    legend to name=mylegend,
    x tick label style={/pgf/number format/1000 sep=\,,font=\small},
    font= \small,
    ]

    \addplot+[solid,thick, smooth,mark size=2pt,mark options={line width=1.0pt},
    mark repeat={4}
    ] table    
    [
    x=N,
    y=l0,
    col sep=space
    ]{figures/data/conv_scal_lambda_de.csv};

    \addplot+[solid,thick, smooth,mark size=2pt,mark options={line width=1.0pt},
    mark repeat={4}
    ] table    
    [
    x=N,
    y=l1,
    col sep=space
    ]{figures/data/conv_scal_lambda_st.csv};

    \addplot+[solid,thick, smooth,mark size=2pt,mark options={line width=1.0pt},
    mark repeat={4}
    ] table    
    [
    x=N,
    y=l2,
    col sep=space
    ]{figures/data/conv_scal_lambda_st.csv};

    \addplot+[solid,thick, smooth,mark size=2pt,mark options={line width=1.0pt},
    mark repeat={4}
    ] table    
    [
    x=N,
    y=l4,
    col sep=space
    ]{figures/data/conv_scal_lambda_st.csv};

  \end{semilogyaxis} 
\end{tikzpicture} &
\begin{tikzpicture}
    \pgfplotsset{
        cycle list/Dark2,
        cycle multiindex* list={
            mark list\nextlist
            Dark2\nextlist
        },
    }

  \begin{semilogyaxis}[
    width=5.4cm, height=5.5cm,     
    grid = major,
    grid style={dashed, gray!30},
    axis background/.style={fill=white},
    xlabel=$N$,
    x tick label style={/pgf/number format/1000 sep=\,},
    font= \small,
    title={{\small{\textbf{DE3}, i.e. $\sigma=1$, $\theta=1$}}},
    ]

    \addplot+[solid,thick, smooth,mark size=2pt,mark options={line width=1.0pt},
    mark repeat={4}
    ] table    
    [
    x=N,
    y=l0,
    col sep=space
    ]{figures/data/conv_scal_lambda_de.csv};

    \addplot+[solid,thick, smooth,mark size=2pt,mark options={line width=1.0pt},
    mark repeat={4}
    ] table    
    [
    x=N,
    y=l1,
    col sep=space
    ]{figures/data/conv_scal_lambda_de.csv};

    \addplot+[solid,thick, smooth,mark size=2pt,mark options={line width=1.0pt},
    mark repeat={4}
    ] table    
    [
    x=N,
    y=l2,
    col sep=space
    ]{figures/data/conv_scal_lambda_de.csv};

    \addplot+[solid,thick, smooth,mark size=2pt,mark options={line width=1.0pt},
    mark repeat={4}
    ] table    
    [
    x=N,
    y=l4,
    col sep=space
    ]{figures/data/conv_scal_lambda_de.csv};
  \end{semilogyaxis} 
\end{tikzpicture}
  &
    \begin{tikzpicture}
    \pgfplotsset{
        cycle list/Dark2,
        cycle multiindex* list={
            mark list\nextlist
            Dark2\nextlist
        },
    }

  \begin{semilogyaxis}[
    width=5.4cm, height=5.5cm,     
    grid = major,
    grid style={dashed, gray!30},
    axis background/.style={fill=white},
    xlabel=$N$,
    legend style={legend pos= south west,font=\small},
    ymax=10,
    x tick label style={/pgf/number format/1000 sep=\,},
    font= \small,
    title={{\small{standard \textbf{sinc}}}},
    ]

    \addplot+[solid,thick, smooth,mark size=2pt,mark options={line width=1.0pt},
    mark repeat={4}
    ] table    
    [
    x=N,
    y=l0,
    col sep=space
    ]{figures/data/conv_scal_lambda_sinc.csv};

    \addplot+[solid,thick, smooth,mark size=2pt,mark options={line width=1.0pt},
    mark repeat={4}
    ] table    
    [
    x=N,
    y=l1,
    col sep=space
    ]{figures/data/conv_scal_lambda_sinc.csv};

    \addplot+[solid,thick, smooth,mark size=2pt,mark options={line width=1.0pt},
    mark repeat={4}
    ] table    
    [
    x=N,
    y=l2,
    col sep=space
    ]{figures/data/conv_scal_lambda_sinc.csv};

    \addplot+[solid,thick, smooth,mark size=2pt,mark options={line width=1.0pt},
    mark repeat={4}
    ] table    
    [
    x=N,
    y=l4,
    col sep=space
    ]{figures/data/conv_scal_lambda_sinc.csv};    
  \end{semilogyaxis} 
\end{tikzpicture}
\\                    
  \multicolumn{3}{c}{
    \ref{mylegend}
  }
\end{tabular}

    \caption{Comparison of $\lambda$ dependence for different quadrature schemes}
    \label{fig:lambda_dependence}
  \end{figure}

  \subsection{A 2d example}
  In order to confirm our theoretical findings in a more complex setting,
  we now look at a 2d model problem with more realistic  data than a single eigenfunction.
  Namely
  we consider the unit square $\Omega=(0,1)^2$. We focus on
  two cases: first we look at what happens if the initial condition does not
  satisfy any compatibility condition, i.e., $u_0 \notin \mathbb{H}^{2\rho}(\Omega)$
  for $\rho \geq 1/4$. The second example is then taken such that
  the data is (almost) in the Gevrey-type class as required by Corollary~\ref{cor:elliptic_full_discretization:super_smooth}
  and Theorem~\ref{thm:convergence_parabolic_super_smooth_full}.
  The inhomogenity in time is taken as $f(t):=\sin(t)u_0$, thus possessing analogous
  regularity properties. We computed the function at $t=0.1$.

  For the discretization in space and of the
  convolution in time of \eqref{eq:parabolic_representation_formula},
  we consider the scheme presented in~\cite{MR20_hp_sinc}. It is based on $hp$-finite elements
  for the Galerkin solver and a $hp$-quadrature on a geometric grid in time for the convolution.
  As it is shown there, such a scheme delivers exponential convergence with respect to the
  polynomial degree and the number of quadrature points. Since we are not interested
  in these kinds of discretization errors, we fixed these discretization parameters
  in order to give good accuracy and only focus on the error due to discretizing the
  functional calculus.

  Since the exact solution is not available, we computed a reference solution with high
  accuracy and compared our other approximations to it. The reference solution
  is computed by the DE1 scheme (as it outperformed the others) by using 8 additional
  quadrature points to the finest approximation present in the graph. As the DE1 scheme
  has finished convergence at this point, we can expect this to be a good approximation.

  We start with the parabolic problem. The initial condition is given by
  $$u_0(x,y):=\omega^{-1} \exp\Big(-\frac{(x-0.5)^2}{\omega}\Big) \exp\Big(-\frac{(y-0.5)^2}{\omega}\Big).$$
  For  $\omega:=1$, this function does not vanish near the
  boundary of $\Omega$ and therefore only satisfies $u_0 \in \mathbb{H}^{1/2-\varepsilon}(\Omega)$.
  We are  in the setting of Lemma~\ref{lemma:parabolic_inhomogeneous_finite_regularity}.
  By inserting $\rho=1/4$ (up to $\varepsilon$) and $r=0$,
  the predicted rates for DE1 and DE2 are roughly $e^{-\frac{6.13}{\sqrt{k}}}$
  and $e^{-\frac{5.62}{k}}$ respectively.
  Figure~\ref{fig:fem_ml1} contains our findings. We observe that all methods
  converge with exponential rate proportional to $\sqrt{\Nquad}$. The double exponential formulas outperforming
  the standard $\sinc$ quadrature. We also observe that picking $\sigma\neq 1$ and $\theta\neq 1$ can
  greatly improve the convergence. The best results being delivered by DE1, i.e. $\sigma=1/2$ and
  $\theta=4$.
  For DE1 and DE2, we observe
  that for a large part of the computation, the scheme outperforms the predicted asymptotic
  rate, but for DE2, the rate appears sharp for large values of $\Nquad$.
    
  As a second example, we used $\omega=0.01$. This function is almost equal to $0$ in a vicinity of
  the boundary  of $\Omega$. Thus we may hope to achieve the improved convergence rate of
  Theorem~\ref{thm:convergence_parabolic_super_smooth_full}. Figure~\ref{fig:fem_ml2}
  shows that it is plausible that the exponential rate of order $\Nquad$ is achieved, although
  the difference is hard to make out. What can be said is that all the double exponential schemes
  greatly outperform the standard $\sinc$ quadrature. The best results are again achieved
  by DE1, which also greatly outperforms the predicted rate for the non-smooth case.

  \begin{figure}
    \centering
    \begin{subfigure}{0.45\linewidth}
      \includeTikzOrEps{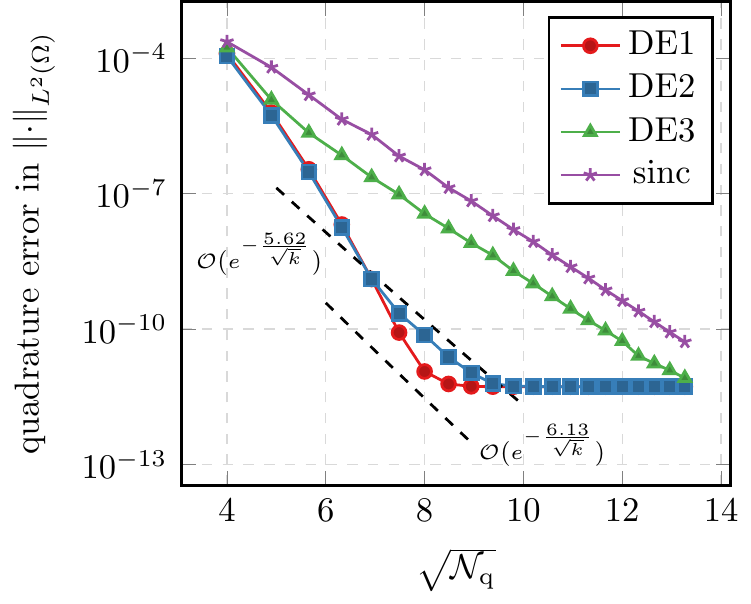}
      \vspace{-\baselineskip}
    \caption{incompatible data}
    \label{fig:fem_ml1}
  \end{subfigure}
  \hfill
    \begin{subfigure}{0.45\linewidth}
      \includeTikzOrEps{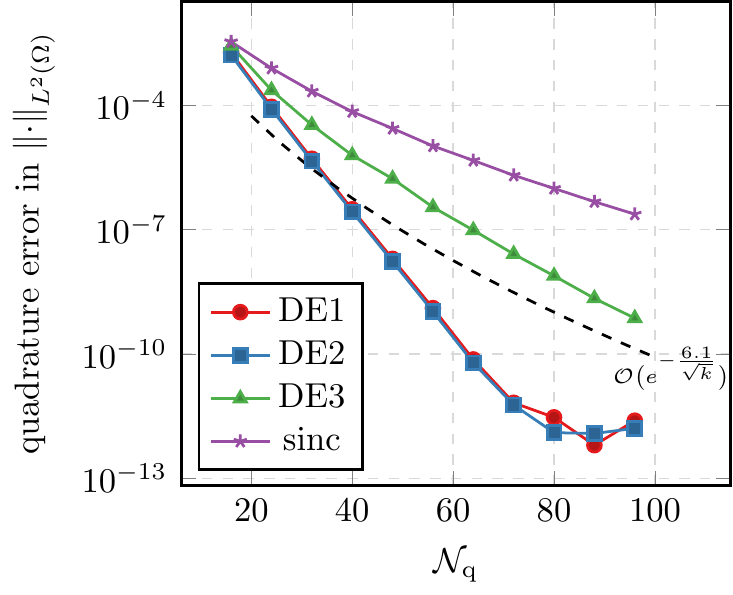}
      \vspace{-\baselineskip}
      \caption{compatible data}
    \label{fig:fem_ml2}
  \end{subfigure}  
  \caption{Comparison of quadrature schemes for 2d in-stationary example; $\alpha=1/\sqrt{2}$,
    $\beta=0.7$.}
\end{figure}

We also looked at the convergence for the elliptic problem. In this class, we 
also included the method based on the Balakrishnan formula.
As a model problem we again used the unit square. We chose $f=1$ as the constant function
and $\beta:=0.4$. We again observe that with $\theta=4$ and $\sigma \in \{0.5,1\}$, the
double exponential formulas significantly outperform the standard $\sinc$ based strategies,
where $\sigma=0.5$ again delivers the best performance.
The predicted rates for DE1 and DE2 in this case are $e^{-\frac{4.7}{\sqrt{k}}}$
  and $e^{-\frac{5.1}{\sqrt{k}}}$ respectively. We observe that asymptotically our
    estimates appear sharp, but with  a large range of values, for which the scheme outperforms
    the predictions.
\begin{figure}[htb]
  \centering
  \includeTikzOrEps{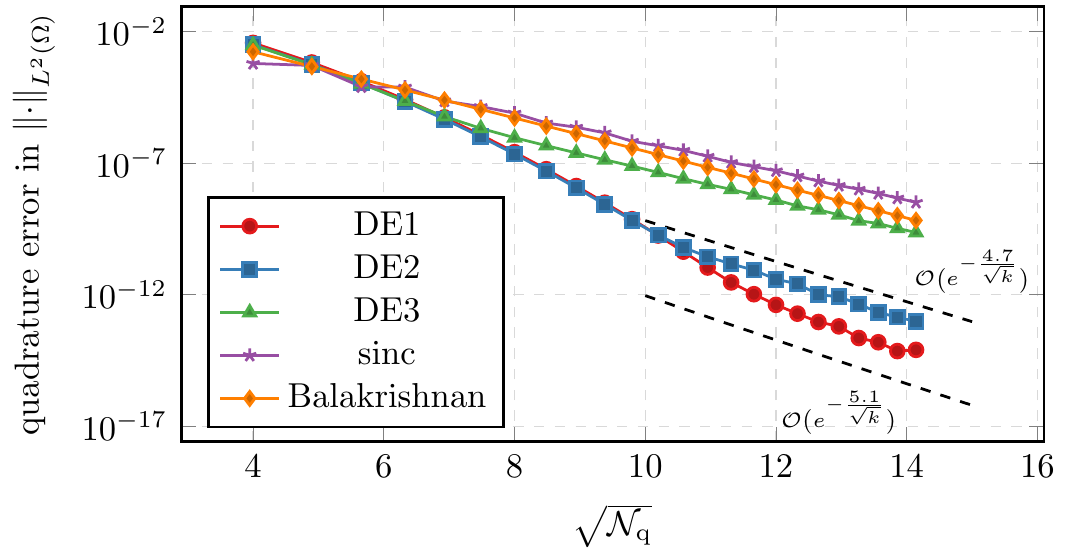}
  \vspace{-\baselineskip}
  \caption{The elliptic problem in 2d, $\beta:=0.4$}
  \label{fig:fem_el}
\end{figure}

  \bibliographystyle{alphaabbr}
\bibliography{literature}

\begin{appendices}  
  \section{Properties of the coordinate transform $\psits$}
  \label{sect:properties_fo_psi}
  In this appendix, we study the transformation $\psits$ in detail,
  as it is crucial to understand the double exponential quadrature scheme.
  Since this transformation is itself defined in a two-part way, we introduce the following nomenclature.

  \begin{definition}
  \label{def:planes}
  We call the complex plane on which $\psits$ is defined the \emph{$y$-plane},
  mainly using the parameter $y$ for its points. The main subset
  of interest there is the strip $\Dt$.

  Using the function $y \mapsto \frac{\pi}{2}\sinh(y)$, the $y$-plane is mapped to the \emph{$w$-plane}.
  The most interesting set is the image of $\Dt$ under this deformation, denoted
  by $\Ht:=\big\{\frac{\pi}{2}\sinh(y), y\in \Dt \big\}$ (named due to its hyperbola shape).

  Finally, using the function $\varphits(w):=\kappa[\cosh(\sigma w)+\ii \theta \sinh(w)]$,
  mapping $\Ht \to \C$, we arrive at the
  \emph{$z$-plane} which corresponds to the range of $\psits$, and the domain
  of the functions used for the Riesz-Dunford calculus.  
  The situation is summarized in Figure~\ref{fig:transformation_planes}.

  If we talk about generic complex numbers without relation to any of the
  specific planes, we use the letter $\zeta$ instead.
\end{definition}

\begin{figure}[htb]
  \centering
  \includeTikzOrEps{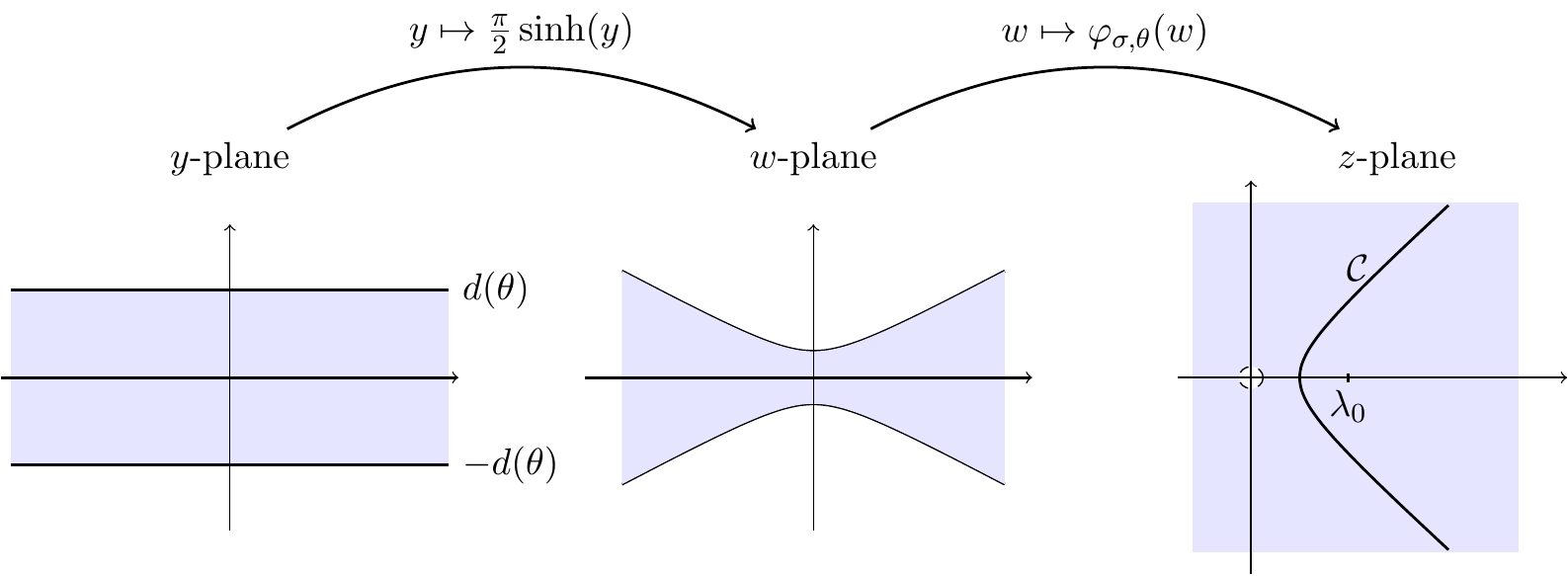}
   \vspace{-\baselineskip}
  \caption{Illustration of the different  planes involved in the mapping $\psits$}
  \label{fig:transformation_planes}
\end{figure}

We start out with some basic properties of $\sinh$.
\begin{lemma}
  \label{lemma:properties_of_sinh}
  The map $y \mapsto \frac{\pi}{2} \sinh(y)$ has the following properties:
  \begin{enumerate}[(i)]
  \item
    \label{it:sinh_is_bijective}    
    It is a bijective mapping $\Dt
    \to \Ht$, see Figure~\ref{fig:transformation_planes}.     
  \item
    \label{it:sinh_cosh_comparable}    
    For $\abs{\Re(\zeta)} \geq \zeta_0 > 0$, there exist  constants
    $c_1, c_2>0$ depending only on $\zeta_0$ such that
    $$
    c_1\abs{\sinh(\zeta)} \leq \abs{\cosh(\zeta)} \leq c_2 \abs{\sinh(\zeta)}.
    $$
  \item
    \label{it:sinh_maps_dexp_to_strip}
    For any $\delta < \pi/2$, 
    $\sinh$ maps the domain $\Dexp$,
    as defined in~\eqref{eq:def_dt_dexp}, to a strip of size $\delta$, i.e.,
    \begin{align}
    \label{eq:dexp_maps_to_strip}
    \abs{\Im(\sinh(y))}\leq \delta \qquad \forall y \in \Dexp.
    \end{align}
  \end{enumerate}
\end{lemma}
\begin{proof}
  Ad (\ref{it:sinh_is_bijective}):
  It is well known that $\sinh$ is injective for $\abs{\Im(y)}<\pi/2$.
  Since $\Ht$ is defined as the range of the map this is sufficient.
%
  Part~(\ref{it:sinh_cosh_comparable}) is an easy consequence of the fact that
  $\sinh$ and $\cosh$ have the same asymptotic behavior for $\Re(\zeta)\to \infty$.
  To see~(\ref{it:sinh_maps_dexp_to_strip}), we estimate for $y \in \Dexp$:
  \begin{align*}
  \abs{\Im(\sinh(y))}
  &=\cosh(\Re(y)) \abs{ \sin(\Im(y))}
    \leq \delta \,\cosh(\Re(y)) e^{-\abs{\Re(y)}}
    \leq \delta.
    \qedhere
  \end{align*}
\end{proof}

\begin{lemma}
  \label{lemma:psi_properties:analytic}
  $\psits$ is analytic on the infinite strip $\Dt$.
  For $d(\theta)$ sufficiently small, both $\psits$ and $\psits'$\
  are non-vanishing on $\Dt$.  
\end{lemma}
\begin{proof}
  The analyticity of $\psits$ is clear.
  In order to analyze the roots, we first rewrite for $w=a+\ii b$,
  separating the real and imaginary parts:
  \begin{equation}
    \begin{aligned}
    \cosh(\sigma w)+\ii \theta \sinh(w)    
      &=\Big(\cosh(\sigma a)\cos(\sigma b)-\theta \cosh(a) \sin(b) \Big)\\
      &+ \ii  \Big(\sinh(\sigma a)\sin(\sigma b)+\theta \sinh(a)\cos(b)\Big).
      \end{aligned}     \label{eq:hyperbola_re_im}
    \end{equation}
    %
%
    We first focus on \textbf{the case  $\mathbf{\boldsymbol \sigma=1}$}. In this case, ~\eqref{eq:hyperbola_re_im} shows that any root $y$ of $\psits$ must satisfy
  for $w:=\frac{\pi}{2}\sinh(y)=:a+b\ii$:
  \begin{align*}
    \cosh(a) \Big(\cos(b)-\theta \sin(b)\Big) &=0
        \quad \text{ and } \quad  \sinh(a) \Big(\theta \cos(b)+\sin(b)\Big) = 0.
  \end{align*}
  Since $\cosh$ has no roots, we get $\cos(b)=\theta \sin(b)$.
  As $\cos(b)=\theta \sin(b)$ and $\theta \cos(b)=-\sin(b)$ is impossible at the same time,
  we get that $a=0$ and $b=\atan(1/\theta) + \ell \pi$ for some $\ell \in \Z$.

  It remains to show that $\frac{\pi}{2}\sinh(y)$ does not map to these points.
  Looking at the real part of $\sinh(y)$ we immediately deduce that if
  $\Im(y) \in (-\pi/2,\pi/2)$, in order to produce a purely imaginary
  result, it must hold that $\Re(y)=0$.
  For the imaginary part, we then get the equation:
  $$
  \sin(\Im(y)) = 2\ell +\frac{2\atan(1/\theta)}{\pi} \quad \text{for some $\ell \in \Z$}
  $$
  which is not possible for $\abs{\Im(y)} \leq d(\theta)
  < \asin(\frac{2\atan(1/\theta)}{\pi})$.
  Next, we show that $\psits'$ also does not vanish. A simple calculation shows
  \begin{align}
    \label{eq:psits_deriv_s1}
  \psito'(y)=\ii \theta \,\psi_{1,\frac{1}{\theta}}(-y)\frac{\pi}{2}\cosh(y).
  \end{align}
  Since the restriction $\theta \geq 1$ was not crucial for the proof,
  $\psi_{1,\frac{1}{\theta}}$ and $\cosh$ have no roots in the symmetric (w.r.t. sign flip) domain $\Dt$.
  This shows that $\psito'$ also is non-vanishing.

  {\textbf{ The case $\mathbf{\boldsymbol \sigma=1/2}$}} is similar, but a little more involved.
  We first show that all zeros of $\cosh(\sigma w) +\ii \theta \sinh(w)$
  satisfy $\Re(w)=0$ and $\abs{\Im(w)} > w_0 >0$ for a constant $w_0$ depending
  only on $\theta$.  
  If $\cosh(w/2)\neq 0$ we can use the double angle formula for $\sinh$ to get  
  \begin{align*}
    0=\cosh(w/2)\big(1 + 2\theta\ii \sinh(w/2)\big)
    \qquad \text { which implies } \qquad 
   \sinh(w/2) = \frac{i}{2\theta}.
  \end{align*}
  Splitting into real and imaginary part, we get for $w=a+\ii b$:
  \begin{align*}
    \sinh(a/2)\cos(b/2) =0 \qquad \text {and } \qquad
    \cosh(a/2)\sin(b/2) = \frac{1}{2\theta}.
  \end{align*}
  If $a\neq 0$, we get that $\cos(b/2)=0$ and thus $\sin(b/2)= \pm 1$.
  This would imply that $\abs{\cosh(a/2)}=\frac{1}{2\theta} < 1/2$
  which is a contradiction. If $a=0$, we get that
  $$
  \abs{b}\geq 2  \big|{\asin\Big(\frac{1}{2\theta}\Big)}\big|> 0.
  $$
  Similarly, one can argue if $\cosh(w/2)=0$, that
  $a=0$ and $b=\pi (2\ell+1)$.
  We then proceed just like in the  case $\sigma=1$ to conclude that $\frac{\pi}{2}\sinh$
  does not map to such points.
  
  In order to investigate its roots, we compute the derivative of $\psith$ as
  \begin{align}
    \label{eq:psits_deriv_s12}
    \psith'(y)
    &=\Big(\frac{1}{2}\sinh\big(\pi \sinh(y)/4\big) + \ii \theta \cosh\big(\pi\sinh(y)/2\big)\Big)\frac{\pi}{2}\cosh(y).
  \end{align}
  Our main concern is when the first bracket reaches zero.
  Substituting $t:=\sinh(\pi\sinh(y)/4)$ and using the double angle formula for $\cosh$ we get
  \begin{align*}
    0&=\frac{t}{2} + \ii \theta (1+2t^2).
  \end{align*}
  Solving this equation, we get that $t$ is purely imaginary and
  for $\theta \geq 1$ satisfies $0<\abs{\Im(t)}<1$.  
  Again writing $w=:a+\ii b$ we get
  \begin{align*}
    \sinh(a/2)\cos(b/2) =0 \qquad \text {and } \qquad
    \cosh(a/2)\sin(b/2) = \Im(t).
  \end{align*}
  Just like we did when showing $\psith\neq 0$ we can argue that
  $a=0$. We get $\sin(b/2)=\Im(t)$. Since $t$ only depends on $\theta$,
  we get that $\abs{b} > b_0 > 0$ with a constant only depending on $\theta$.
  We proceed as  when showing $\psith\neq 0$ to conclude that
  $\psith'$ has no root in $\Dt$ for $d$ sufficiently small (depending on $\theta$).
\end{proof}


Next, we study the growth of $\psits(y)$ as $\abs{\Re(y)} \to \infty$.
\begin{lemma}
  \label{lemma:psi_growth}
  There exist constants $c_1, c_2$,  $\gamma_1, \gamma_2 >0$ such that
  for $y \in \Dt$ we can estimate
  \begin{align}
    \label{eq:psi_growh}
    c_1 \exp(\gamma_1 e^{\abs{\Re(y)}}) \leq \abs{\psits(y)} \leq c_2 \exp(\gamma_2 e^{\abs{\Re(y)}}),
  \end{align}
  i.e., the growth of $\psi$ is double exponential.
  Additionally, we can bound
  \begin{align}
    \label{eq:bound_dpsi}
    |{\psits'}(y)|\lesssim |\psits(y)|\cosh(\Re(y)).  
  \end{align}
\end{lemma}
\begin{proof}
  We start with the simple case $\boldsymbol{\sigma}\mathbf{=1}$. And focus on what happens if
  $\abs{\Re(y)} > y_0>0$ for a to be tweaked constant $y_0$. The  values with $0\leq \abs{\Re(y)}\leq y_0$ can
  be covered by adjusting $c_1$ and $c_2$, due to the compactness
  of the set $\{y \in \Dt: \abs{\Re(y)} \leq y_0\}$ and the fact that $\abs{\psits(y)}>0$
  by Lemma~\ref{lemma:psi_properties:analytic}. 
  We first only consider the  $y-w$ part of the transformation.
  We compute, since $\frac{1}{2}<\cos(t)\leq 1$ for $t \in (-1/2,1/2)$:
  \begin{align*}
    \abs{\Re(\sinh(y))}&= \abs{\sinh(\Re(y))} \cos(\Im(y)) \sim \abs{\sinh(\Re(y))}
                         \sim c_1 e^{\abs{\Re(y)}}.
  \end{align*}
  Easy calculation shows
  that for $h_1(\eta):=\abs{\cos(\eta) - \theta \sin(\eta)}$
  and $h_2(\eta):=\abs{\theta  \cos(\eta) + \sin(\eta)}$ it
  holds that $[h_1(\eta)]^2+[h_2(\eta)]^2=1+\theta^2$.
  Thus, for $w \in \Ht$
  with $\abs{\Re(w)}\geq 1$, we can calculate:
  \begin{equation}
    \label{eq:growth_phi}
  \begin{aligned}
    \abs{\cosh(w) + \ii \theta \sinh(w)}^2
    &= \abs{\cosh(\Re(w))}^2[h_1(\Im(w))]^2+ \abs{\sinh(\Re(w))}^2 [h_2(\Im(w))]^2  \\
    &\gtrsim  \min\big(\cosh(\Re(w)),\abs{\sinh(\Re(w))}\big)^2(1+\theta^2) 
      \;\gtrsim \; e^{2\abs{\Re(w)}}.
    \end{aligned}
  \end{equation}
  Overall, we see the lower bound of~(\ref{eq:psi_growh}).
  The upper bound is easily seen, as $\abs{\sinh(y)}$ and $\abs{\cosh(y)}$
  both grow  exponentially and the bound only depends on the real part of the argument.
  \eqref{eq:bound_dpsi} follows from~\eqref{eq:psits_deriv_s1} and the
  asymptotics \eqref{eq:growth_phi}and~\eqref{eq:psi_growh}.

  We now look at how to adapt the proof to \textbf{the case $\boldsymbol\sigma\mathbf{=1/2}$}.
  If $\abs{\Re(w)} \geq 4\ln(2)$, we get
  \begin{multline}
    \label{eq:psi_growth_w}
    \abs{\cosh(w/2) + \ii \theta \sinh(w)} \\
    \geq \theta \abs{\sinh(w)} - \abs{\cosh(w/2)} 
    \geq \frac{1}{2}\Big(e^{\abs{\Re(w)}} - 2 - e^{\abs{\Re(w)/2}}\Big)  
    \geq \frac{1}{4} e^{\abs{\Re(w)}}.
  \end{multline}
  The argument for the $y-w$-transformation stays the same. The upper bound
  also follows easily from the triangle inequality and the growth
  of $\sinh$ and $\cosh$.

  To see~(\ref{eq:bound_dpsi}),
  we combine  \eqref{eq:psits_deriv_s12} with
  the asymptotic estimate~\eqref{eq:psi_growth_w} to get
  \begin{align*}
    |{\psits'}(y)|
    &\lesssim
      e^{\pi\abs{\Re(\sinh(y))}/2} \cosh(\Re(y)) \lesssim |\psits(y)|\cosh(\Re(y)).  \qedhere
  \end{align*} 
\end{proof}

While on the full strip $\Dt$, the image of the  transformation is
difficult to study, the restriction to a certain subdomain is much better behaved.
\begin{lemma}
    \label{lemma:psi:properties:right_halfplane}
    For $\sigma=1$, there exists a constant $\delta>0$, depending on $\theta$,
    such that restricted to the domain $\Dexp$, 
    $\psito$ maps to a sector in the right-half plane,
    $$
    S_{\omega}:=\big\{z \in \C: \abs{\operatorname{Arg}(z)}\leq\omega  \big\}
    \qquad \text{with}  \quad \omega<\frac{\pi}{2}.
    $$
    For $\sigma=1/2$, and for all $\varepsilon >0$, there exists a constant
    $\delta>0$, depending on $\theta$ and $\varepsilon$, such that restricted to the domain
    $\Dexp$ the transformation $\psith$ maps to the sector $S_{\pi/2+\varepsilon}$.
    
    In both cases, there exist constants $c_1, c_2>0$ such that
    $\psits$ satisfies for all $\lambda >
    \lambda_0 > \kappa:$
    \begin{align}
      \label{eq:psi:properties:resolvent_right_halfplane}
      \abs{\psits(y)-\lambda}\geq c_1 \quad  \text{and}\quad
      \abs{\psits'(y)(\psits(y) - \lambda)^{-1}} \leq c_2 \cosh(\Re(y))
      \qquad \forall y \in \Dexp,
    \end{align}
    where $c_1, c_2$ only depend on $\lambda_0$ and $\theta$.
  \end{lemma}
  \begin{proof}
    By Lemma~\ref{lemma:properties_of_sinh}(\ref{it:sinh_maps_dexp_to_strip}) it is sufficient
    to consider the mapping of $\varphits$ restricted to small strips around the real axis.
    We start with the simpler \textbf{case $\boldsymbol \sigma\mathbf{=1}$.}
    Going back to (\ref{eq:hyperbola_re_im}) and  writing
    $w:=\frac{\pi}{2} \sinh(y)=:a+\ii b$, we note that if $\abs{b}$ is sufficiently small, we can guarantee
    that ${\cos(b)- \theta\sin(b)} > c >0$
    for some constant $c >0$ depending on $\theta$.
  We have
  \begin{align*}
  \Re(\varphits(w))&=\kappa \cosh(a)\big(\cos(b)-\theta \sin(b)\Big) > c \,  \kappa \cosh(a), \\
  \abs{\Im(\varphits(w))}&=\kappa \abs{\sinh(a)}|\theta \cos(b)+ \sin(b)| \leq  (1+\theta) \kappa \abs{\sinh(a)}.
  \end{align*}
  Which easily gives
  $$
  0\leq\frac{\abs{\Im(\varphits(w))}}{\Re(\varphits(w))}\leq (1+\theta) c^{-1}
  \qquad \forall w\in \C,\; \abs{\Im(w)} \text{sufficiently small}. 
  $$

  Next, we show that \textbf{for} $\boldsymbol{\sigma=1/2}$, sufficiently thin strips in the $w$-domain are mapped to
  sectors with opening angle $\pi/2+\varepsilon$.
  Such sectors are characterized by 
  $$
  -\Re(\varphits(w)) \leq \widetilde{\varepsilon}
  \abs{\Im(\varphits(w))} \qquad \forall \abs{\Im(w)}<b_0(\widetilde{\varepsilon}) 
  $$
  for $\widetilde{\varepsilon} > 0$ depending on $\varepsilon$.
  The interesting case is $\Re(\varphits(w))<0$. There, we get for $\abs{b}\leq \pi$:
  \begin{align*}
    -{\Re(\varphits(w))}
    &=-\kappa \cosh(a/2)\cos(b/2)+\kappa \theta \cosh(a) \sin(b)  \leq \kappa  \theta \cosh(a) \abs{\sin(b)}.
  \end{align*}
  For the imaginary part, the double angle-formula gives:
  \begin{align*}
    \abs{\Im(\varphits(w))}
    &\geq \kappa \theta \abs{\sinh(a) \cos(b)} - \kappa \abs{\sinh(a/2)\sin(b/2)} \\
    &\geq \frac{\kappa \theta}{2}  \abs{\sinh(a) \cos(b)} + \kappa \abs{\sinh(a/2)} \big(\theta\cosh(a/2)\abs{\cos(b)} - \abs{\sin(b/2)}\big).
  \end{align*}
  For $\abs{b}$ sufficiently small, we get $\theta \cos(b) - \abs{\sin(b/2)} >0$, and thus
  the last term is non-negative. We conclude
  \begin{align*}
    -{\Re(\varphits(w))}
    &\leq 2\abs{\frac{\cosh(a)}{\sinh(a)}  \frac{\sin(b)}{\cos(b)}} \abs{\Im(\varphits(w))}.
  \end{align*}
  For $a>1$, the first term is uniformly bounded. By making $b$ sufficiently small, we can ensure
  $\abs{\sin(b)}/{\cos{b}}< \widetilde{\varepsilon}$.
  For $a<1$, we note that by restricting the values of $b$ it can be easily seen
  that $\Re(\varphits(w))\geq 0$. Thus we can conclude that $\psits$  maps to the stated sectors.

  Next, we prove the bounds on the distance to the real axis,
  again primarily investigating the behavior of $\varphits$ in thin strips.
  We focus on {$\boldsymbol{\sigma=1/2}$}.
  Since  $ \varphits(0)=\kappa < \lambda_0$  and $\varphits$ is continuous, there exist
  constants $0<q<1$ and $\widetilde{\delta}>0$ such that $\abs{\varphits(w)} < q\lambda_0$
  for all $\abs{w}<\widetilde{\delta}$.
  This gives
  $$
  \abs{ \varphits(w)-\lambda} \geq \lambda -  \abs{\varphits(w)} \geq \lambda_0 -\abs{\varphits(w)}
>(1-q)  \lambda_0  \qquad \forall \abs{w}\leq \widetilde{\delta}.
$$
By selecting $\delta < \widetilde{\delta}/2$ in the definition of $\Dexp$, we may then
continue by only considering $\abs{a}:=\abs{\Re(w)}> \widetilde{\delta}/2$.
  The imaginary part of $\varphith(y)$ satisfies:
  \begin{align*}
     \Im(\varphith(y))
    &=\kappa \sinh(a/2)\sin(b/2)+\kappa \theta \sinh(a)\cos(b)  \\
    &= \kappa \sinh(a/2)\Big(\sin(b/2) + 2\theta \cosh(a/2) \cos(b) \Big).
  \end{align*}
  For $\abs{b} \leq \pi/4$ we have $\sin(b/2) + 2\cos(b) > 0$
  and thus can conclude that $|{\Im(\varphith(y))}| \geq c > 0 $
  and in turn $|{\varphith(y) - \lambda}| \geq c > 0$.
  The case $\sigma=1$ follows similarly, but not using the double angle formula.
  
  To see that $\psits'(y)(\psits(y)-\lambda)^{-1}$  can also be uniformly bounded,
  we only need to focus on large values of $y$ (and therefore $w$).
  Asymptotically, we estimate for $\abs{b}<\pi/4$:
  \begin{align*}
    \abs{\Im(\psits(y))}
    &\geq
      -\kappa \abs{\sinh(\sigma a)}\abs{\sin(\sigma b)}
      +\kappa \theta \abs{ \sinh(a)\cos(b)}
      \gtrsim
      \abs{\sinh(a)}
      \!\!\stackrel{\text{\eqref{eq:growth_phi} or \eqref{eq:psi_growth_w}}}{\gtrsim} \!\! \abs{\psits(y)}.
  \end{align*}
 Using (\ref{eq:bound_dpsi}) then concludes the proof.
%
\end{proof}

In order to apply the double exponential formulas for the Riesz-Dunford calculus,
it is important to understand where $\psits(z)$ hits the real line. We start with
the $w$-domain.
\begin{lemma}
  \label{lemma:psi_hits_lambda_w_domain}
  Fix $\lambda \geq \lambda_0 > 1$. Then the following holds
  for every  $w\in \C$ with $\Re(w)\neq 0$ and
    \begin{align}
      \label{eq:pole_in_w_domain}
    \cosh(\sigma w) + \ii \theta  \sinh(w) = \lambda:
    \end{align}
  \begin{enumerate}[(i)]
  \item
    \label{it:psi_hits_lambda_w_domain:bounds}
    There exist  constants $c_1,c_2,c_3>0$ such that $w$ satisfies
    \begin{align*}
        \log(\lambda) - c_1 \leq \abs{\Re(w)}
        &\leq \log(\lambda) + c_1 \quad \text{and}\quad
          \abs{\Im(w)} \geq
          \begin{cases}
            \atan(\theta) & \text{if }\sigma=1 \\
            \max\big(\frac{\pi}{2} - \frac{c_2}{\theta \sqrt{\lambda}},c_3) &  \text{if } \sigma=1/2
          \end{cases},
    \end{align*}
    where $c_1$ depends on $\lambda_0$ and $\theta$, $c_2$ depends on $\lambda_0$, and
    $c_3$ depends on $\theta$.
  \item
    \label{it:psi_hits_lambda_w_domain:number}
    Given $0<r<R$, 
    the number  $N_w(\lambda,r, R)$ of points $w$ satisfying (\ref{eq:pole_in_w_domain})
    with $r\leq \abs{\Im(w)}\leq R$  is  bounded uniformly in $\lambda$
    by
    $$
    N_w(\lambda,r,R) \leq C \abs{R-r}
    $$
    The constant $C$ depends only on $\theta$.  
\item
  \label{it:psi_hits_lambda_w_domain:distance}
  There exist at most four values $p_1,\dots,p_4$ depending on $\lambda$, $\theta$, and $\sigma$
  such that all points satisfying~(\ref{eq:pole_in_w_domain}) can be written
  as
  \begin{align}
    \label{eq:poles_w_explicit}
    w=p_j +   \frac{2\ell}{\sigma} \pi \ii
    \qquad \text{ for some $\ell \in \Z$, $j\in \{1,4\}$.}
  \end{align}
  If $w$ solves \eqref{eq:pole_in_w_domain} then $-\overline{w}$ does as well.
  %
%
%
  \end{enumerate}
\end{lemma}
\begin{proof}
  We start with the simpler case $\sigma=1$.
  By separating the real and imaginary part as in~\eqref{eq:hyperbola_re_im}, we can observe that the
  critical points $w=a+\ii b$ with $a\neq 0$ are located at
  \begin{align}
    \cosh(a)&= \frac{\lambda}{\sqrt{1+\theta^2}},
              \qquad b=-\atan(\theta) + 2\ell \pi , \quad \ell \in \Z.
              \label{eq:w_pos_of_pole}
  \end{align}
  This implies that $\abs{a} \sim \ln(\lambda)$, and we also see that
  for each $\ell$, there are two such points, one in each half-plane. 
  All the statements follow easily. Note that in (\ref{it:psi_hits_lambda_w_domain:distance})
  only two families are needed.

  For the remainder of the proof we therefore focus on the case $\sigma=1/2$.
  \textbf{Ad~(\ref{it:psi_hits_lambda_w_domain:bounds}):}
  We start with the bound on the real part and write
  $w=a+\ii b$.
  We note that if $\abs{a} > \max\big(1,2\ln\big(\frac{8}{ \theta}\big)\big)$
  one can estimate using elementary considerations that
  $e^{\abs{a}}/4 \leq \abs{\sinh(w)}$ and
  $e^{\abs{a}/2}/\theta\leq e^{\abs{a}}/8$.

  We then calculate:
  \begin{align*}
    \frac{e^{\abs{a}}}{4}
    &\leq \abs{\sinh(w)}
      =\frac{1}{\theta} {\abs{\lambda - \cosh(w/2)}}
    \leq \frac{\lambda}{\theta} + \frac{e^{\frac{\abs{a}}{2}}}{\theta} 
    \leq \frac{\lambda}{\theta} + \frac{e^{\abs{a}}}{8}.
  \end{align*}
  From this, the statement readily follows.
  The other direction is shown similarly:
  \begin{align*}
    e^{a}&\geq \abs{\sinh(w)}
           =\frac{1}{\theta} {\abs{\lambda - \cosh(w/2)}}
           \geq \frac{\lambda}{\theta} - \frac{e^{\frac{\abs{a}}{2}}}{\theta}
           \geq \frac{\lambda}{\theta} - \frac{e^{\abs{a}}}{8}.
  \end{align*}
  For $\abs{a} \leq \max\big(1,2\ln\big(\frac{8}{ \theta}\big)\big)$, we use the bound
  $\abs{\varphits(w)}\lesssim e^{\abs{a}}$ to see that  $\ln(\lambda)$ must be uniformly bounded.
  The final bound on the real part of $w$ then follows for
  $c_1:=\max\big(\ln(8/\theta),\ln(9\theta/8),\widetilde{c}\big)$, where $\widetilde{c}$ is used to compensate for the case of small $a$.

  Looking at the imaginary part of equation~\eqref{eq:pole_in_w_domain}, we get
  from the double-angle formulas
  \begin{align}
    0&=\sinh(a/2)\sin(b/2) + \theta \sinh(a)\cos(b)  \nonumber\\
     &=\sinh(a/2)\Big(\!\sin(b/2)+2\theta \cosh(a/2) \big(1-2\sin^2(b/2)\big)\!\Big).
       \label{eq:condition_pole_in_w_domain_imag}
  \end{align}
  Since we assume $a\neq 0$, we get
  by substituting $\tau:=\sin(b/2)$
  \begin{align}
    0=\tau+2\theta \cosh(a/2) (1-2\tau^2) \qquad \text{ or }\qquad
    \tau=\frac{1 \pm \sqrt{1+32 \theta^2 \cosh^2(a/2)}}{8 \theta \cosh(a/2)}.
        \label{eq:poles_eq_for_sinb}
  \end{align}
  From this, using the asymptotic behavior
  $$
  \tau = \pm \frac{1}{\sqrt{2}} + \bigO\Big(\frac{1}{\theta \cosh(a/2)}\Big)
  \quad \text{ and } \quad
  \asin\Big(\frac{1}{\sqrt{2}} + h\Big) = \frac{\pi}{4} + \bigO(h)
  $$
  the statement follows for large $a$, since $b=2\asin(\tau)$ and $\cosh(a/2)\gtrsim \sqrt{\lambda}$
  by (\ref{it:psi_hits_lambda_w_domain:bounds}).
  For small $a$, we note that~\eqref{eq:poles_eq_for_sinb}
  shows that $\Im(w)\neq 0$. By continuity, it
  must therefore hold that $\abs{\Im(w)}>0$ uniformly.

  \textbf{Ad~(\ref{it:psi_hits_lambda_w_domain:number}) and (\ref{it:psi_hits_lambda_w_domain:distance}):}
  We square the defining equation~\eqref{eq:pole_in_w_domain}, getting
  \begin{align}
    \lambda^2
    &=\cosh^2(w/2)+2\ii \theta \cosh(w/2)\sinh(w)-\theta^2 \sinh^2(w) \nonumber\\
    &=\cosh^2(w/2)+4\ii \theta \cosh^2(w/2)\sinh(w/2)-4\theta^2 \sinh^2(w/2)\cosh^2(w/2).
  \end{align}
  Writing $\cosh^2(w/2)=1+\sinh^2(w/2)$, we get that
  $t:=\sinh(w/2)$ solves the quartic equation
  \begin{align}
    \lambda^2
    &=1+t^2+4\ii \theta (1+t^2)t-4\theta^2 t^2 - 4\theta^2 t^4.
      \label{eq:psi_hits_lambda_w_domain:quartic_eqn}
  \end{align}  
  This means there
  can be at most 4 such values $t_{0}, \dots t_4$ for any $\lambda$ 
  and it must hold that
  \begin{align*}
    \sinh(w/2)&=t_j \quad  \text{or }  \quad w=w_j + 4\pi \ell \ii  
  \end{align*}
  Here $w_j$ is the solution with $\Re(w_j)>0$ and minimal value of $\abs{\Im(w_j)}$.
  (\ref{it:psi_hits_lambda_w_domain:distance}) and (\ref{it:psi_hits_lambda_w_domain:number}) follow.
  The fact that $-\overline{w_j}$ also solves~\eqref{eq:pole_in_w_domain} follows by conjugating both sides of
  the equation~\eqref{eq:psi_hits_lambda_w_domain:quartic_eqn}.
\end{proof}

Next, we show that points which have positive distance to the poles in the
$w$-plane are mapped to points with distance $\lambda$ in the $z$-plane.
Note that in the following Lemma we include some additional points in order to avoid distinguishing more cases. We also exclude most of the imaginary axis, as for small values of $\lambda$ it might
contain poles which are structurally different than the ones involving large $\lambda$.
\begin{lemma}
  \label{lemma:phi_keeps_distance}
  Define
  $$b_0:=\max\big\{b\geq 0: \quad \abs{\varphits(\ii \tau)}\leq (\lambda_0 + \kappa)/2 \quad\forall \abs{\tau}\leq b\big\}.$$
  Fix $\lambda \geq \lambda_0 > \kappa$ and
  define the set
  \begin{align*}
    \mathcal{M}_{\lambda}
    &:=\Big\{ p_\lambda + i \ell \pi, \; \text{with }\; p_\lambda \in \C \text{ such that } \varphits(p_\lambda) = \lambda
      \text{ and }\ell \in \Z \Big\}
      \cup \{ \ii b : \; \abs{b}\leq b_0 \}.
  \end{align*}
  For any $\delta > 0$, there exists a constant $c(\delta)>0$, depending only on $\delta$ and $\theta$, such that 
  for all $w \in \C$ with $\operatorname{dist}(w,\mathcal{M}_\lambda)> \delta$ we can estimate
  $$
  \abs{ \varphits(w) - \lambda} \geq c(\delta) \lambda.
  $$
\end{lemma}
\begin{proof}
  We first deal with the issues close to the imaginary axis.
  Since $\abs{\varphits(\ii b)}\leq(\lambda_0+\kappa)/2$ if $\abs{b}\leq b_0$, we can find
  $\varepsilon \in (0,\delta/2)$ such that 
  \begin{align*}
    \abs{\varphits(\widetilde{w})} < 3\lambda_0/4 + \kappa/4
    \quad \text{for all} \quad
    \widetilde{w} \in U_{\varepsilon}
    &:=\big\{
    \abs{\Re(\widetilde{w})} \leq \varepsilon \; \text{and}\;\abs{\Im(\widetilde{w})} \leq b_0 \big\}.
  \end{align*}
  If $\abs{\Re(w)}\leq \varepsilon$, the distance condition to $\mathcal{M}_{\lambda}$ implies
  $\abs{\Im(w)} \leq b_0 - \delta/2$ and thus $w \in U_{\varepsilon}$.
  We can therefore calculate:
  \begin{align*}
    \abs{\varphits(w)-\lambda} \geq \lambda - \abs{\varphits(w)}
    \geq \frac{\lambda}{4}\Big(1 - \frac{\kappa}{\lambda_0}\Big).
  \end{align*}
  Next, we deal with  small values of $\lambda$.
  For constants $\Lambda$, $C_w$ to be fixed later,
  consider $\lambda \in [\lambda_0, \Lambda]$
  and define the set
  $$\mathcal{R}:=\{w \in \C: \; \abs{\Re(z)}\geq \varepsilon, \abs{w} \leq C_w\}.$$
  We show that the stated bound holds for $w \in \mathcal{R}$.
  By the Lemma~\ref{lemma:psi_hits_lambda_w_domain}(\ref{it:psi_hits_lambda_w_domain:distance}), 
  the points of $\mathcal{M}_\lambda$ can be written as
  $$
  \mathcal{M}_\lambda=\Big\{ p_{j} +  \pi \ell  \ii , \;\; \ell \in \Z,\; j=1,\dots,4 \Big\}
  \cup  \Big \{ \ii b : \; \abs{b}\leq b_0 \Big\}
  $$
  for some reference points $p_{j} \in \C$, w.l.o.g. $\abs{\Im(p_j)}\leq \pi$.
  We introduce set
  $$
  \mathcal{M}_{\Lambda,C_w}:=\Big\{p_{j} + \ell \pi \ii, \; \abs{\ell} \leq L, \, j=1,\dots,4 \Big\}
  $$
  where $L:=\lceil (C_w+1)/\pi\rceil$ is taken large enough that for all $\lambda \in [\lambda_0, \Lambda]$,
  it holds that
  $$
  \mathcal{M}_{\lambda} \cap \mathcal{R} \subseteq \mathcal{M}_{\Lambda,C_w},
  $$
  i.e., it contains all the poles of size less or equal than $C_w$ uniformly in $\lambda$
  (but possibly also some larger ones).  
  We consider the map
    $$
    \Phi(\lambda, w):= \big(\varphits(w) - \lambda\big)^{-1}\prod_{p_\lambda \in \mathcal{M}_{\Lambda,C_w}}{(w - p_\lambda)}.
    $$
    By the inverse function theorem, the points  $p_{j}=\varphits^{-1}(\lambda)$ depend continuously on $\lambda$ since
    $\varphits'\neq 0$ away from the imaginary axis (where we stay away from by construction). Thus, also the other points
    $p_{\lambda}$ depend continuously on $\lambda$.
    Similarly, the denominator only has simple zeros for $w\in \mathcal{M}_\lambda$. Since, in that case the numerator also vanishes one can argue that $\Phi$ has a  continuous
    extension to $[\lambda_0, \Lambda]\times \mathcal{R}$ which is bounded, i.e., it holds
    $$
    \abs{\frac{\prod_{p_\lambda \in \mathcal{M}_{\Lambda,C_w}}{(w - p_\lambda)}}{\varphits(w) - \lambda}} \leq C(\Lambda,C_W)
    \qquad \text{or}\qquad 
    \abs{\varphits(w) - \lambda} \geq \frac{1}{C(\Lambda,C_W)} \prod_{p_{\lambda} \in \mathcal{M}_{\Lambda,C_w}}{\abs{w - p_\lambda}}.
    $$
    Thus, if $w$ is separated from $\mathcal{M}_\lambda$ and the imaginary axis,
    so is $\varphits(w)$ from $\lambda$.
    If $\lambda \in [0,\Lambda]$ and $\abs{w} > C_w:=\max(\log(2\Lambda/c_1),4\ln(2))$,
    (where $c_1$ is the constant in \eqref{eq:growth_phi} or \eqref{eq:psi_growth_w})
    we get:
    $$
    \abs{\varphits(w)-\lambda} \geq \abs{\varphits(w)}-\lambda
    \quad\;\stackrel{\mathclap{\eqref{eq:growth_phi}, \eqref{eq:psi_growth_w}}}{\geq}\quad\;
    c_1 e^{\abs{\Re(w)}} - \Lambda 
    \geq \Lambda \geq\lambda.
    $$
    We therefore may from now on assume that $\lambda$ is sufficiently large as we see fit.
    In preparation for the rest of the proof, we note that for $\zeta, \mu \in \R$,
    w.l.o.g., $\abs{\zeta}\leq \abs{\mu}$:
    \begin{align}
      \label{eq:difference_of_cosh}
      \abs{\cosh(\mu) - \cosh(\zeta)}
      &=\cosh(\abs{\mu})-\cosh(\abs{\zeta})
        =\int_{\abs{\zeta}}^{\abs{\mu}}{\sinh(\tau)\,d\tau}
        \geq \sinh(\abs{\zeta}) (\abs{\mu}-\abs{\zeta}).
    \end{align} 
    
  Because it is much simpler, we start with the \textbf{case $\boldsymbol \sigma\mathbf{=1}$}.
  We note that in this case $\mathcal{M}_\lambda$ consists of the points mapped to $\pm \lambda$.
  We distinguish three cases, depending on whether $\Re(w)$ is small and if
  $\Im(w)$ is close to a pole or not. 
    \paragraph{Case 1:} \emph{$(1+\theta)\kappa \cosh(\Re(w))< \lambda/2:$}
      The triangle inequality gives:
      \begin{align*}
        \abs{\kappa[\cosh(w)+ \ii \theta \sinh(w)] - \lambda}
        &  \geq \lambda - (1+\theta) \kappa \cosh(\Re(w)) 
        \geq \frac {\lambda}{2}.
      \end{align*}      
      \paragraph{Case 2:}
      \emph{$2(1+\theta)\kappa \cosh(\Re(w)) \geq  \lambda$ and
      there exists a point $p \in \mathcal{M}_\lambda$ with
      $\abs{\Im{w} - \Im{p}} \leq \varepsilon_1$
      for $\varepsilon_1$ sufficiently small.}
      We note that this implies that 
      $\abs{\Re(w)-\Re(p)}$ is positive.
      Due to the symmetry in \eqref{eq:poles_w_explicit}
      we may in addition assume that $\operatorname{sign}(\Re(w))=\operatorname{sign}(\Re(p))$.
      Writing $h(\eta):=\cos(\eta)-\theta\sin(\eta)$, 
      we note that
      $\abs{h(\Im(w))} > c> 0$ by \eqref{eq:w_pos_of_pole} and the fact that adding $\pi \ell$ might only change the sign. 
      If $h(\Im(p)) \geq 0$, we consider the real part of $\varphits(w)-\lambda$ to get: 
      \begin{multline*}
        \abs{\cosh(\Re(w))h(\Im(w))) - \lambda}\\
        \begin{aligned}
         &= \abs{\cosh(\Re(w))h(\Im(p)) - \lambda
            +\cosh(\Re(w))\big(h(\Im(w))-h(\Im{p})\big)} \\
          &\geq  \abs{\cosh(\Re(w))h(\Im{p}) - \lambda}
          - \cosh(\Re(w))\abs{h(\Im(w))-h(\Im{p})} \\
          &\stackrel{\mathclap{\eqref{eq:difference_of_cosh}}}{\geq} \min\big(\abs{\sinh(\Re(w))},\abs{\sinh(\Re(p))}\big) \abs{\Re(w) - \Re(p)}
          - 2\cosh(\Re(w))\varepsilon_1 \\
          &\gtrsim \lambda          
        \end{aligned}
      \end{multline*}
      where in the last step we chose $\varepsilon_1$ sufficiently small
      (but independent of $\lambda$).
      If $h(\Im(p))\leq 0$,  by continuity we can enforce $h(\Im(w))\leq 0$
      as long as $\varepsilon_1$ is sufficiently small.      
      The necessary calculation then is  even easier because $\varphits(w)$ maps to the left 
      half-plane.
      \paragraph{Case 3:} \emph{ $2(1+\theta) \cosh(\Re(w)) \geq  \lambda$ and
      $\abs{\Im(w) - \Im(p)} \geq \varepsilon_1 > 0$ for all $p \in \mathcal{M}_\lambda$.}
      We estimate imaginary part of $\varphits$.
      Since $\Im(w)$ has positive distance from the points
      in \eqref{eq:w_pos_of_pole}, we get
      $\abs{\sin(\Im(w))+\theta \cos(\Im(w))}>c>0$.
      Which in term gives
      $$
      \abs{\sinh(a)}\abs{\sin(\Im(w))+\cos(\Im(w))}
      \geq  c\abs{\sinh(\Re(w))} \gtrsim \lambda
      $$
      where the last part only holds for large enough
      cases of $\Re(w)$ not covered by Case~1.

    Now we show, how the proof has to be adapted for the \textbf{case $\boldsymbol \sigma \mathbf{=1/2}$},    
    again focusing on the asymptotic case of large $\lambda$.
    By Lemma~\ref{lemma:psi_hits_lambda_w_domain}(\ref{it:psi_hits_lambda_w_domain:bounds}), all the points
    $p \in \mathcal{M}_\lambda$ satisfy $\abs{\Re(p)} \sim \log(\lambda)$.
    Looking at the imaginary part of the defining
    equation for $p$ we get that
    \begin{align*}
      \abs{\cos(\Im(p))}&=\abs{\frac{\sinh(\Re(p)/2)}{\theta \cosh(\Re(p))} \sin(\Re(p))}
                      \lesssim \frac{e^{-a/2} }{\theta} \lesssim \lambda^{-1/2}.
    \end{align*}
    Thus, for any $\varepsilon_2 >0$,
    assuming $\lambda$ is sufficiently large,  all the points $p \in \mathcal{M}_\lambda$ satisfy $\cos(p)< \varepsilon_2$.

    We again have to distinguish three cases:
    
    \paragraph{Case 1:} \emph{$(1+\theta) \cosh(\Re(w))< \lambda/2:$}
        One can argue just like in the $\sigma=1$ case.
        \paragraph{Case 2:}
        \emph{$2(1+\theta) \cosh(\Re(w)) \geq  \lambda$ and
      there exists a point $p \in \mathcal{M}_\lambda$ with
      $\abs{\Im{w} - \Im{p}} \leq \varepsilon_1$
      for $\varepsilon_1$ sufficiently small.} We note that this implies that $\abs{\Re(w)-\Re(p)}$ is positive.
    Since $\abs{\cos(w)}< \varepsilon_2$, we note that $\abs{\sin(w)} > 1-\varepsilon_2$.
    By possibly adding $\ii \ell \pi$, we can write
    \begin{align}
      \label{eq:write_lambda_using_phi}
    \lambda=-\theta\cosh(\Re(p))\sin(\Im(p+\ell \ii\pi)) + \cosh(\Re(p/2))\sin(\Im(p+\ell \ii \pi)/2).
    \end{align}
    For large values of $\lambda$, the $\cosh(\Re(p))$ term
    is dominating. Therefore, we have $\operatorname{sign}(\Re(\varphits(p)))=-\operatorname{sign}(\sin(\Im(p)))$.
    By continuity we can  enforce that $\operatorname{sign}(\sin(\Im(w)))=\operatorname{sign}(\sin(\Im(p)))$.
    Since for the case $\Re(\varphits(w))\leq 0$ the statement is trivial,
    we  only have to consider the case $\operatorname{sign}(\sin(\Im(w)))=-1$ or $\ell=0$ in \eqref{eq:write_lambda_using_phi}.    
    We look at the real part of $\varphits(w)-\lambda$ to get: 
      \begin{align*}
        \big|-\theta\cosh(&\Re(w))\sin(\Im(w)) - \lambda + \cosh(\Re(w)/2)\cos(\Im(w)/2)\big|\\
        &\geq\begin{multlined}[t][12cm]
          \abs{-\theta\cosh(\Re(w))\sin(\Im(p)) - \theta \cosh(\Re(p)) \sin(\Im(p))} \\
         -\abs{\cosh(\Re(w))\big(\sin(\Im(w))-\sin(\Im{p})\big)} 
         - \bigO\big(\cosh(\Re(w)/2)\big)
         \end{multlined}
         \\
          &\gtrsim \min\Big(\abs{\sinh(\Re(w))},\abs{\sinh(\Re(p))}\Big) \abs{\Re(w) - \Re(p)}
          - C\cosh(\Re(w))\varepsilon_1 \\
          &\gtrsim \lambda                  
      \end{align*}
      where we absorbed the term $\cosh(\Re(w)/2)$ into $\cosh(\Re(w))\varepsilon_1$ by assuming $\lambda$
      sufficiently large and in the last step we chose $\varepsilon_1$ sufficiently small
      (but independent of $\lambda$).
    \paragraph{Case 3:} \emph{ $2(1+\theta) \cosh(\Re(w)) \geq  \lambda$ and
      $\abs{\Im{w} - \Im{p}} \geq \varepsilon_1 > 0$ for all $p \in \mathcal{M}_\lambda$.}
      Since all the points $p \in \mathcal{M}_\lambda$ satisfy
      $\abs{\cos(\Im(p))}<\varepsilon_2$ and for every zero of
      $\cos$, there exists a value $p \in \mathcal{M}_\lambda$ with $\Im(p)$ close to it,
      this means that $\abs{\cos(\Im(w))}\geq \delta_2$ for a constant $\delta_2$ depending only on $\varepsilon_1$ and $\varepsilon_2$.      
      We estimate
      $$
      \abs{\Im(\varphits(w))}
      \geq \theta\abs{\sinh(a)}\abs{\cos(\Im{w})} - \abs{\cosh(w/2)}
      \gtrsim \abs{\sinh(\Re(w))} \gtrsim \lambda
      $$
      where the last part only holds for large enough
      cases of $\Re(w)$ not covered by the first case. 
\end{proof}

\begin{lemma}
  \label{lemma:psi_hits_lambda}
  Fix $\lambda \geq \lambda_0 > \kappa$.
  Consider the set $$P_\lambda^y:=\Big\{ y \in \Dt: \psits(y)=\lambda \Big\},
  $$
  where $d(\theta)$ is taken sufficiently small.
  Then the following holds:
  \begin{enumerate}[(i)]
  \item
    \label{it:psi_hits_lambda:finite_number}
    For any $\nu_0 \in [0,d(\theta)]$ and $\delta > 0$
    there are only finitely many points $y_1,\dots,y_{N_p(\lambda,\delta)} \in P_\lambda^y$ satisfying
    $$
    \nu_0-\frac{\delta}{\ln(\lambda/\kappa)}<\abs{\Im(y_\ell)} < \nu_0+\frac{\delta}{\ln(\lambda/\kappa)}
    \qquad \forall \ell \in 1,\dots, N_p(\lambda,\delta)
    $$
    The number $N_p(\lambda,\delta)$  of such points  can be bounded
    by a  constant depending only on $\delta$, $\theta$, $\sigma$ and $d(\theta)$,
    but independently of $\lambda$ and $\nu_0$ . 
  \item
    \label{it:psi_hits_lambda:scale_of_imag_part}
    For $\lambda \in (\lambda_0,\Lambda)$, one can bound
    $ \displaystyle 
    \abs{\Im(y)}\geq \gamma(\Lambda) > 0
    \; \forall y \in P^y_\lambda,
    $
    with a constant $\gamma(\Lambda)$ depending
    on $\Lambda$, $\theta$, $\kappa$, $\lambda_0$.       
    For $\lambda$ sufficiently large, the following asymptotic holds:
    \begin{align}
      \label{eq:psi_hits_lambda:scale_of_imag_part}
      \abs{\Im(y)}\geq 
        \begin{cases}          
          \frac{\atan(\theta)}{\ln(\lambda/\kappa)}   - \bigO\big(\frac{1}{\ln^2(\lambda/\kappa)}\big)
          & \text{if $\sigma=1$},\\
          \frac{\pi}{2\ln(\lambda/\kappa)}  - \bigO\big(\frac{1}{\ln^2(\lambda/\kappa)}\big)  &\text{if $\sigma=1/2$}
        \end{cases}\qquad \forall y \in P^y_\lambda,
    \end{align}    
    where  the implied constants depend only on $\theta$, $\kappa$, $\lambda_0$ .
  \item
    \label{it:psi_hits_lambda:viable_path}
    There exists a parameter $d_{\lambda} \in (d(\theta)/2,d(\theta)]$ and a constant $c>0$ such that
    \begin{align}
      \abs{\psits(a \pm \ii d_{\lambda}) - \lambda}
      \geq c \,\kappa \lambda^{-1} \qquad \forall a \in \R.
    \end{align}
    $c$ depends on $\theta$, $\sigma$  and $\lambda_0$ but is independent of $\lambda$.
  \end{enumerate}
\end{lemma}
\begin{proof}
  For now, consider the case $\kappa=1$ and $\lambda_0>1$.
  Since $\psits(0)=1$, due to continuity, there exists a factor $d(\theta)>0$
  and $\varepsilon>0$ such that
  $$
  \abs{\psits(y)}< 1 + \varepsilon<\lambda_0<\lambda \qquad \forall \abs{y} \leq d(\theta).
  $$
  By taking $d(\theta)$ at last this small in the definition of $\Dt$ we can exclude the
  poles satisfying $\Re(w)=\Re(\sinh(y))=0$.
  
  \textbf{Ad~(\ref{it:psi_hits_lambda:finite_number}):}
  Since $\sinh$ is injective by
  Lemma~\ref{lemma:properties_of_sinh}(\ref{it:sinh_is_bijective}) we only need
  to count the points in $H_{\theta}$ which are mapped to $\lambda$ by $\varphits$.  
  Consider a point $w=a+\ii b$ in the $w$ domain and let  $\frac{\pi}{2}\sinh(y)=w$ with $y=\xi+\ii \nu \in \Dt$.
  Then
  \begin{align*}
    \abs{a}=\frac{\pi}{2}\abs{\sinh(\xi)}\abs{\cos(\nu)} 
                \quad \text{ and }\quad
   \abs{b}&=\frac{\pi}{2}\abs{\cosh(\xi)}\abs{\sin(\nu)}. 
  \end{align*}
  Simple computations give $\abs{a} \tan(\abs{\nu}) \leq \abs{b} \leq (\pi/2+\abs{a}) \tan(\abs{\nu})$,
  or, by inserting  Lemma~\ref{lemma:psi_hits_lambda_w_domain}(\ref{it:psi_hits_lambda_w_domain:bounds})
  \begin{align}
    \label{eq:psi_hits_lambda:simple_bound}
    \big(\ln(\lambda)-c_1\big) \tan(\abs{\nu}) \leq \abs{b} \leq (\pi/2+c_1+\ln(\lambda)) \tan(\abs{\nu}). 
  \end{align}
  For the length $L$ of this interval, we compute for $\nu_\pm:=\nu_0 \pm \delta \ln(\lambda/\kappa)^{-1}$
  (without changing the number of points, we may assume $0\leq \nu_- \leq \nu_+ \leq d(\theta)$):
  \begin{align*}
    L&=\ln(\lambda)\big(\tan(\abs{\nu_+})-\tan(\nu_-)\big)+
       (\pi/2+c_1) \tan(\nu_+) + c_1\tan(\nu_-)\\
     &\leq \ln(\lambda) \int_{\nu_-}^{{\nu_+}}{\frac{1}{\cos^2(\tau)}d\tau} + \pi+ 4c_1 
       \leq 8 \delta + \pi+ 4c_1
  \end{align*}
  where in the last step we used $\cos(\tau)>1/2$ for $\abs{\tau}<1/2$ and the definition
  of $\nu_\pm$.
  Since the length of this interval is bounded uniformly in $\lambda$, we
  can apply ~\ref{lemma:psi_hits_lambda_w_domain}(\ref{it:psi_hits_lambda_w_domain:number}) to bound
  $N_p(\lambda,\delta)$.

  \textbf{Ad~(\ref{it:psi_hits_lambda:scale_of_imag_part}):}
  We only show the asymptotics.
  Fix $y\in \Dt$ and write $w:=\frac{\pi}{2}\sinh(y)$.
  For $0<\abs{y}<1/2$ it can be easily seen that $\abs{y}+\frac{2}{3} \abs{y}^3> \tan(\abs{y}).$
  Now, if $\abs{\Im(y)}\geq 2\ln(\lambda)^{-1}$ there is nothing left to show. Otherwise, we can bound
  \begin{align*}
    \abs{\Im(y)}
    &\geq \abs{\tan(y)} - \frac{4}{3 \ln(\lambda)^{3}}
    \stackrel{\mathclap{\eqref{eq:psi_hits_lambda:simple_bound}}}{\geq }
      \frac{\abs{\Im(w)}}{1+c_1 + \ln(\lambda)}
      - \frac{4}{3 \ln(\lambda)^{3}}
     \geq \frac{\abs{\Im(w)}}{\ln(\lambda)} - \bigO\Big(\frac{1}{\ln(\lambda)^2}\Big).
    \end{align*}
    The result follows from Lemma~\ref{lemma:psi_hits_lambda_w_domain}(\ref{it:psi_hits_lambda_w_domain:bounds}).
    
  \textbf{Ad~(\ref{it:psi_hits_lambda:viable_path}):}
  For $d_{\lambda}=d(\theta)$, we can not guarantee that
  $\psi(\xi+\ii\,d_\lambda)$ does not hit the value $\lambda$. In this case,
  we have to modify $d_\lambda$ slightly to get robust estimates.
  For $d\in \R$, consider the hyperbolas
  \begin{align}
    \gamma_{d}(\xi):=\frac{\pi}{2}\sinh(\xi+\ii\,d) \quad  \xi \in \R.
  \end{align}
  
  In the light of Lemma~\ref{lemma:phi_keeps_distance} we need to ensure
  that $\operatorname{dist}(\gamma_{d_\lambda},w_p) \gtrsim 1$ for all $w_p \in \mathcal{M}_\lambda$.
  We will be looking for $d_\lambda$ in  a small strip around $d(\theta)$. To simplify notation
  we define the length
  $$
  \omega:=d(\theta)  \frac{\ln(\lambda_0)}{2\ln(\lambda)}
  \quad \text{ such that }  \quad d(\theta)/2 \leq d(\theta)-\omega \leq d(\theta).
  $$
  
  To make things symmetric
  with respect to the real axis,
  we consider $\widetilde{\mathcal{M}}_{\lambda}:=\mathcal{M}_{\lambda}- {\mathcal{M}}_{\lambda}$.
  It will therefore be sufficient to focus on the upper right quadrant of the complex plane.
  All other cases follow by symmetry.
  
  We write $\widetilde{\mathcal{M}}^y_\lambda:=\asinh(\frac{2}{\pi}\widetilde{\mathcal{M}}_{\lambda})$ for the corresponding points in $y$-domain.
  We start by noting that we can easily stay away from the problematic
  parts of the imaginary axis by making $d(\theta)$ sufficiently small,
  as if $\abs{\Re(\sinh(y))}<\varepsilon$ we have $\abs{\Im(\sinh(y))}<(1+\varepsilon)\sin(\Im(y))$; thus
  for small real parts we can ensure to fit between $(-b_0,b_0)$ on the imaginary axis. This
  also means that we can safely assume $\abs{\Re(w_\lambda)}>\varepsilon >0$ since our path will have
  already positive distance to other possible poles.

  By~(\ref{it:psi_hits_lambda:finite_number}),
  the number of points $y_\lambda$ in $\widetilde{\mathcal{M}}^y_\lambda$
  in the strip 
  $d(\theta)-\omega \leq \Im(y_\lambda)\leq d(\theta)$ can
  be bounded by a constant $N$, independent of $\lambda$.
  In order to also avoid points in $\widetilde{\mathcal{M}}_\lambda^y$ which are
  close but outside the critical strip  we also avoid the boundary points $d(\theta)-\omega$
  and $d(\theta)$.
  Since $N+2$ strips of width $\frac{\omega}{2N+4}$ can not cover a strip of width $\omega$,
  there exists a value $d_{\lambda}$ such that
  $$
  d(\theta)-\omega \leq d_{\lambda} \leq d(\theta)
  \quad \text{and} \quad
  \abs{\Im(y_{\lambda}) - d_{\lambda}} \geq \frac{\omega}{2(N+2)} \quad \forall y_\lambda \in \widetilde{\mathcal{M}}^y_\lambda.
  $$
  For ease of notation, we define  $\delta:=\ln(\lambda_0)/(4N+8)$  
  and note that $\abs{\Im(y_{\lambda}) - d_\lambda}\geq \delta \ln(\lambda)^{-1}$.  
  We show that
  \begin{align*}
    \operatorname{dist}(\gamma_{d_\lambda},w_p)
    &\geq \mu > 0
      \qquad \forall w_p \in \widetilde{\mathcal{M}}_\lambda
  \end{align*}
  for a constant $\mu > 0$ independent of $\lambda$.
  We fix $y_p:=\xi_{p}+  \ii \nu_{p} \in \widetilde{\mathcal{M}}_{\lambda}^y$ 
  and a point on $\gamma_{d_\lambda}$  denoted by $y_{\gamma}=\xi_{\gamma}+ d_\lambda \ii$, .
  We write $s_{p}:=\operatorname{sign}(\nu_p-d_{\lambda})$ and
  distinguish two cases:
  $(\xi_p - \xi_\gamma)s_{p} \leq 0$ and $(\xi_p - \xi_\gamma)s_{p}>0$.
  For symmetry reasons, we only consider the case $\xi_\gamma,\xi_p,\nu_{p} >0$.
  \textbf{If $\boldsymbol{(\xi_p - \xi_\gamma)s_{p} \geq 0}$}, we get:
    \begin{align*}
      \Im(s_{p}(\sinh(y_p)-\sinh(y_{\gamma})))
      &=s_{p} \big(\cosh(\xi_p)\sin(\nu_p) - \cosh(\xi_{\gamma})\sin(d_{\lambda}) \big)\\
      &=\cosh(\xi_p)s_{p}\big(\sin(\nu_p)-\sin(d_{\lambda})\big)
        + \underbrace{s_p \Big(\cosh(\xi_p)-\cosh(\xi_{\gamma})\Big)\sin(d_\lambda)}_{\geq 0} \\  
      &\geq \cosh(\xi_p)  s_{p} \int_{d_{\lambda}}^{\nu_p}{\cos(\tau) d\tau} 
      \gtrsim \cosh(\xi_p) \frac{\delta}{\ln(\lambda)}.
    \end{align*}
    For the \textbf{case $\boldsymbol{(\xi_p - \xi_\gamma)s_{p} <0}$}, we calculate:
      \begin{align*}
        -s_{p}\Re(\sinh(y_{p})-\sinh(y_\gamma))
        &=-s_{p}\sinh(\xi_p)\cos(\nu_p) + s_p\sinh(\xi_\gamma)\cos(d_{\lambda}) \\
        &=-s_{p}\sinh(\xi_p)\Big(\cos(\nu_p)-\cos(d_{\lambda})\big) + 
          \underbrace{s_{p}\Big(\sinh(\xi_\gamma)-\sinh(\xi_p)\Big)\cos(d_\lambda))}_{\geq 0}
        \\
        &\geq \sinh(\xi_p) s_{p} \int_{d_\lambda}^{\nu_p}{\sin(\tau) d\tau} 
        \gtrsim \sinh(\xi_p)
          \sin(d(\theta)/2)
          \frac{\delta}{\ln(\lambda)}.
      \end{align*}
      Since $\cosh(\xi_p) \geq \sinh(\xi_p) \gtrsim \min(\ln(\lambda),\varepsilon)$
      we can conclude that
    $$
    \operatorname{dist}(\gamma_{d_\lambda},w_\lambda)
    \gtrsim \min\Big(\ln(\lambda) \frac{1}{\ln(\lambda)}, \varepsilon \Big) \geq \mu > 0.
    $$
    We can now apply Lemma~\ref{lemma:phi_keeps_distance} to get to the final result.
    The general case $\kappa \neq 1$ follows by dividing the equation $\psits(y)=\lambda$ by $\kappa$.
    We can therefore just replace $\lambda$ by $\lambda/\kappa$ in all statements.
  \end{proof}

\end{appendices}

\end{document}